\documentclass[11pt,reqno,tbtags]{amsart}

\usepackage{amsthm,amssymb,amsfonts,amsmath, dsfont}
\usepackage{amscd} 
\usepackage{mathrsfs}
\usepackage[numbers, sort]{natbib}
\usepackage{subfigure, graphicx}
\usepackage{vmargin}
\usepackage{setspace}
\usepackage{paralist}
\usepackage[pagebackref=true]{hyperref}
\usepackage{color}
\usepackage[normalem]{ulem}
\usepackage{stmaryrd} 
\usepackage[refpage,noprefix,intoc]{nomencl} 
\usepackage{bbm} 

\providecommand{\R}{}
\providecommand{\Z}{}
\providecommand{\N}{}

\renewcommand{\R}{\mathbb{R}}
\renewcommand{\Z}{\mathbb{Z}}
\renewcommand{\N}{{\mathbb N}}

\newcommand{\p}[1]{{\mathbb P}\left(#1\right)}
\newcommand{\psub}[2]{{\mathbb P}_{#1}\left(#2\right)}

\newcommand{\I}[1]{{\mathds 1}_{#1}}

\newcommand{\Esub}[2]{{\mathbb E_{#1}}\left[#2\right]}

 \newcommand{\bag}{\begin{align}}
\newcommand{\bags}{\begin{align*}}
\newcommand{\eag}{\end{align*}}
\newcommand{\eags}{\end{align*}}

\newtheorem{thm}{Theorem}[section]
\newtheorem{lem}[thm]{Lemma}
\newtheorem{prop}[thm]{Proposition}
\newtheorem{cor}[thm]{Corollary}

\newtheorem*{remark}{Remark}

\providecommand{\ora}[1]{}
\renewcommand{\ora}[1]{\overrightarrow{#1}}

\newcommand{\eqdist}{\ensuremath{\stackrel{\mathrm{d}}{=}}}

\hypersetup{
    bookmarks=true,         
    unicode=false,          
    pdftoolbar=true,        
    pdfmenubar=true,        
    pdffitwindow=true,      
    pdftitle={My title},    
    pdfauthor={Author},     
    pdfsubject={Subject},   
    pdfnewwindow=true,      
    pdfkeywords={keywords}, 
    colorlinks=true,       
    linkcolor=blue,          
    citecolor=blue,        
    filecolor=blue,      
    urlcolor=blue           
}

\reversemarginpar
\definecolor{clou}{rgb}{0.8,0.25,0.5125}

\begin{document}

\title{The spreading speed of solutions of the non-local Fisher-KPP equation}
\author{Sarah Penington}
\address{Mathematical Institute, University of Oxford, Woodstock Road, Oxford, OX2 6GG, UK}
\email{penington@maths.ox.ac.uk}

\begin{abstract} 
We consider the Fisher-KPP equation with a non-local interaction term.
In~\cite{hamel2014}, Hamel and Ryzhik showed that in solutions of this equation, the front location at a large time $t$ is $\sqrt 2 t +o(t)$.
We study the asymptotics of the second order term in the front location.
If the interaction kernel $\phi(x)$ decays sufficiently fast as $x\to \infty$ then this term is given by $-\frac{3}{2\sqrt 2 }\log t +o(\log t)$, which is the same correction as found by Bramson in~\cite{bramson1983} for the local Fisher-KPP equation. 
However, if $\phi$ has a heavier tail then the second order term is $-t^{\beta +o(1)}$, where $\beta \in (0,1)$ depends on the tail of $\phi$.
The proofs are probabilistic, using a Feynman-Kac formula.
Since solutions of the non-local Fisher-KPP equation do not obey the maximum principle, the proofs differ from those in~\cite{bramson1983}, although some of the ideas used are similar.
\end{abstract}

\date{\today}
\maketitle

\section{Introduction}
We shall study solutions of the non-local Fisher-KPP equation
\begin{equation} \label{eq:(star)}
\frac{\partial u}{\partial t}=\tfrac{1}{2}\Delta u +\mu u (1- \phi \ast u), \quad t>0, \quad x\in \R,
\end{equation}
where $\mu>0$, $\phi\in L^1(\R)$ is non-negative, and 
$$\phi \ast u(t,x)=\int_\R \phi(y)u(t,x-y)dy.
$$ 
This equation is used to model non-local interaction and competition in a population.
It can be seen as a generalisation of the (local) Fisher-KPP equation
\begin{equation} \label{eq:(dagger)}
\frac{\partial u}{\partial t}=\tfrac{1}{2}\Delta u +\mu u (1- u), \quad t>0, \quad x\in \R,
\end{equation}
introduced in \cite{fisher1937} and \cite{kolmogorov1937}.

In both equations~\eqref{eq:(star)} and~\eqref{eq:(dagger)}, we can think of $u(t,x)$ as the population density at location $x$ at time $t$.
Then the Laplacian terms represent the diffusive motion of the population, and the non-linear terms give the rate of change of the population density due to birth and death.
This rate should be proportional to the population density multiplied by the amount of available resources.
In~\eqref{eq:(dagger)}, we are supposing that the resources available at $x$ at time $t$ are only depleted by the population at $x$.
However, consumption of resources is not completely local, and modelling the depletion of resources by the spatial average $\phi \ast u$, as in~\eqref{eq:(star)}, may be more realistic.
Another way of viewing~\eqref{eq:(star)} and~\eqref{eq:(dagger)} is for the variable $x$ to represent the value of some trait that varies in the population, e.g.~height. Then we think of $u(t,x)$ as the population density with height $x$ at time $t$.
In this case, the Laplacian terms represent the incremental change in height due to mutations, and in~\eqref{eq:(dagger)}, we are assuming that the resources available for individuals with height $x$ are only depleted by other individuals with height $x$. In fact, although competition may be stronger between individuals with similar heights, individuals with different heights may still be competing for resources, making~\eqref{eq:(star)} a better model for this situation.  
See \cite{britton1989}, \cite{genieys1989} and Section 2.3 of~\cite{volpert2009} for background on these and other biological motivations for studying the non-local Fisher-KPP equation.

The main mathematical interest in the non-local Fisher-KPP equation comes from studying the similarities and differences in the behaviour of solutions of the local and non-local equations. In particular, we shall be interested in how this depends on $\phi$ and $\mu$.

Several authors \cite{britton1990,gourley2000,gourley1993,genieys1989,nadin2011} have studied the behaviour of solutions of \eqref{eq:(star)} using numerical simulations and asymptotic stability analysis.
Others \cite{berestycki2009,fang2011,alfaro2012,hamel2014} have proved rigorous results on the properties of travelling wave  and steady state solutions of \eqref{eq:(star)}.
The introduction of \cite{hamel2014} gives a summary of the rich behaviour suggested by these results.

We shall consider the long time behaviour of solutions to the initial value problem
\begin{equation} \label{nonlocal_fkpp}
\begin{cases}
\frac{\partial u}{\partial t}=\tfrac{1}{2}\Delta u +u (1- \phi \ast u), \quad t>0, \quad x\in \R, \\
u(0,x)=u_0(x), \quad x\in \R,
\end{cases}
\end{equation}
where $u_0\geq 0$ and $u_0\in L^\infty (\R)$.
We shall assume throughout (as in~\cite{hamel2014}) that $\phi$ satisfies 
\begin{equation} \label{eq:phi_cond}
\phi \in L^1 (\R), \,\phi \geq 0,\,
\int_{-\infty}^\infty \phi(x)dx =1 \text{ and }\phi\geq \eta \text{ a.e.~on }(-\sigma,\sigma)\text{ for some }\eta, \sigma>0. 
\end{equation}
Note that for $\phi_1\in L^1(\R)$ with $\phi_1\geq 0$ and $\int_{-\infty}^\infty \phi_1 (x)dx=a$, if $v_1$ solves~\eqref{nonlocal_fkpp} with $\phi=\phi_1$ then $v_2:=av_1$ solves~\eqref{nonlocal_fkpp} with $\phi=\phi_2:=a^{-1}\phi_1$. Since $\int_{-\infty}^\infty \phi_2 (x)dx=1$, our assumptions will be satisfied by $v_2$ (if $\phi_1$ satisfies the last condition in~\eqref{eq:phi_cond}).
The last condition in~\eqref{eq:phi_cond} is biologically reasonable, since $\phi(x)$ represents the amount of interaction between individuals at displacement $x$. 
Note also that we have taken $\mu=1$ in~\eqref{eq:(star)}; if $u$ solves~\eqref{eq:(star)} then by rescaling to consider $v(t,x):=u(\mu^{-1}t,\mu^{-1/2}x)$ we have a solution of~\eqref{nonlocal_fkpp} (with rescaled $\phi$).

By standard arguments for parabolic equations, the solution to \eqref{nonlocal_fkpp} exists for all $t>0$, is smooth and classical on $(0,\infty)\times \R$ and satisfies 
\begin{equation} \label{eq:u_bound}
0\leq u(t,x)\leq e^t \|u_0\|_\infty\, \,\forall t\geq 0, x\in \R
\end{equation}
(see Section 3 of \cite{hamel2014}).
We shall take initial conditions $u_0$ which are compactly supported on the right, i.e.~there exists some $L>0$ such that $u_0(x)=0$ $\forall x\geq L$. We will be interested in the front location of $u(t,\cdot)$ for large times $t$, i.e.~the location $x>0$ beyond which $u(t,\cdot)$ is $o(1)$.

The main result in the literature on the spreading speed of solutions is Theorem~1.3 in~\cite{hamel2014}, which was proved using PDE methods and shows that the front location is $\sqrt 2 t +o(t)$ for large times $t$.
The following is (a slightly strengthened form of) Theorem~1.3 in~\cite{hamel2014}.
\begin{thm} \label{thm:speed}
Suppose $u_0\geq 0$, $u_0\in L^\infty (\R)$, $u_0\not \equiv 0$ and
there exists $L<\infty$ such that $\|u_0\|_\infty \leq L$ and $u_0 (x)=0$ $\forall x\geq L$. Let $u$ denote the solution of \eqref{nonlocal_fkpp}.
There exists $m^*>0$ such that for any $\epsilon>0$,
\begin{align*}
\liminf_{t\to\infty}\inf_{x\in[0,(\sqrt 2-\epsilon) t]}u(t,x)&\geq m^*
\quad \text{ and }\quad 
\lim_{t\to\infty}\sup_{x\geq \sqrt 2 t}u(t,x)=0.
\end{align*}
\end{thm}
The tools we will develop to study solutions of~\eqref{nonlocal_fkpp} will allow us to give
an alternative proof of this result using probabilistic methods. We shall also be able to find the asymptotics of the $o(t)$ term in the front location.

There is a well-known result on the asymptotics of the front location for the local Fisher-KPP equation.
In Theorem 3 of \cite{bramson1983}, Bramson showed that if $u$ is a solution to \eqref{eq:(dagger)} with $\mu=1$, if $\int_0^\infty y e^{\sqrt 2 y}u(0,y)dy<\infty$ and if for some $x_0>0$, $\inf_{x\leq -x_0}u(0,x)>0$ then
\begin{equation} \label{eq:bram1}
u(t,x+m(t))\to w(x)
\end{equation}
uniformly in $x$ as $t\to \infty$,
where 
\begin{equation} \label{eq:bram2}
 m(t)=\sqrt 2 t -\tfrac{3}{2\sqrt 2}\log t+\mathcal O (1)
\end{equation}
and $w(x)\to 0$ as $x\to \infty$, $w(x)\to 1$ as $x\to -\infty$.
(In fact, $w$ is a travelling wave solution of~\eqref{eq:(dagger)}.)
In particular, if $r(t)\to \infty$ as $t\to \infty$, then
\begin{align} \label{eq:bramsonfkpp} 
\lim_{t\to\infty}\sup_{x\geq \sqrt 2 t -\frac{3}{2\sqrt 2}\log t+r(t)}u(t,x)=0
\,\,\text{ and }\,\,
\lim_{t\to\infty}\inf_{x\leq \sqrt 2 t-\frac{3}{2\sqrt 2}\log t-r(t)}u(t,x)=1.
\end{align}
One of the main tools in the study of the local Fisher-KPP equation is a maximum principle: if $u^1$ and $u^2$ are two solutions of \eqref{eq:(dagger)} with $0\leq u^1(0,x)\leq u^2(0,x)\leq 1$ $\forall x\in \R$
then $0\leq u^1(t,x)\leq u^2(t,x)\leq 1$ $\forall t>0, x\in \R$
(see Proposition 3.1 in \cite{bramson1983}).
Bramson's proof of~\eqref{eq:bram1} and~\eqref{eq:bram2} uses a combination of the maximum principle and probabilistic methods using the Feynman-Kac formula.

The lack of an equivalent maximum principle for the non-local Fisher-KPP equation makes it considerably less tractable and means that we cannot apply Bramson's methods directly.
Instead, we rely solely on the Feynman-Kac formula
to prove a version of \eqref{eq:bramsonfkpp} for solutions to the non-local Fisher-KPP equation, as long as $\phi(x)$ decays sufficiently fast as $x\to \infty$ (more precisely, as long as $\limsup_{r\to \infty}r^\alpha \int_r^\infty \phi(x)dx<\infty$ for some $\alpha>2$).
We shall see different behaviour
if instead $\limsup_{r\to \infty}r^\alpha \int_r^{Kr} \phi(x)dx >0$ for some $\alpha\in (0,2)$ and $K<\infty$.

\subsection{Main results}
We now give precise statements of our main results.
Let $u$ denote the solution to \eqref{nonlocal_fkpp}.
We shall assume throughout that $u_0\geq 0$, $u_0\not \equiv 0$, $u_0\in L^\infty (\R)$ and $u_0$ is compactly supported on the right.
In particular, we define $L\in(0,\infty)$ such that $\|u_0\|_\infty \leq L$ and $u_0 (x)=0$ $\forall x\geq L$.
Also, we take $\eta$, $\sigma>0$ and suppose that $\phi$ satisfies assumption~\eqref{eq:phi_cond} with this choice of $\eta$ and $\sigma$.

The first result is a Bramson-type logarithmic delay result which shows that if $\phi(x)$ decays sufficiently fast as $x\to\infty$ then the front location in the non-local Fisher-KPP equation has similar behaviour to the local Fisher-KPP equation.
\begin{thm} \label{thm:lighttail}
Suppose that there exists $\alpha>2$ such that for $r_0$ sufficiently large,  
\begin{equation*} 
\forall r\geq r_0, \,\int_r^\infty \phi(x)dx \leq r^{-\alpha}.
\end{equation*}
Then there exist $A<\infty$ and $m^*>0$ such that
\begin{align*}
\liminf_{t\to\infty}\inf_{x\in[0,\sqrt 2 t-\frac{3}{2\sqrt 2}\log t-A(\log \log t)^3]}u(t,x)&\geq m^*\\
\text{ and }\quad 
\lim_{t\to\infty}\sup_{x\geq \sqrt 2 t -\frac{3}{2\sqrt 2}\log t+10\log \log t}u(t,x)&=0.
\end{align*}
\end{thm}
The next two results show different behaviour to the local Fisher-KPP equation if $\phi$ has a heavier tail.
The first of these results gives a lower bound on the front location.
\begin{thm} \label{thm:heavylower}
Suppose that there exists $\alpha\in (0,2)$ such that for $r_0$ sufficiently large, 
\begin{equation*} 
\forall r\geq r_0, \,\int_r^\infty \phi(x)dx \leq r^{-\alpha}.
\end{equation*}
Then there exists $m^*>0$ such that for any $\beta>\frac{2-\alpha}{2+\alpha}$, 
$$
\liminf_{t\to\infty}\inf_{x\in[0,\sqrt 2 t -t^\beta]}u(t,x)\geq m^*.
$$
\end{thm}
The next result is an upper bound on the front location.
\begin{thm} \label{thm:heavyupper}
Suppose that there exist $\alpha\in (0,2)$ and $K<\infty$ such that for $r_0$ sufficiently large, 
\begin{equation*} 
\forall r\geq r_0, \,\int_r^\infty \phi(x)dx \leq r^{-\alpha/2}\text{ and }\int_r^{Kr}\phi (x) dx\geq r^{-\alpha}.
\end{equation*}
Then for any $\beta<\frac{2-\alpha}{2+\alpha}$,
$$
\lim_{t\to\infty}\sup_{x\geq \sqrt 2 t -t^\beta}u(t,x)=0.
$$
\end{thm}
\begin{remark}
Suppose there exist $\alpha\in (0,2)$,
$0<C_1\leq C_2<\infty$ and $r_0<\infty$
 such that $\forall r\geq r_0$, $C_1 r^{-\alpha}\leq \int_r^\infty \phi (x)dx\leq C_2 r^{-\alpha}$. 
Let $\beta = \frac{2-\alpha}{2+\alpha}$. 
Then by Theorems~\ref{thm:heavylower} and~\ref{thm:heavyupper}, there exists $m^*>0$ such that for any $\epsilon>0$,
$$
\liminf_{t\to\infty}\inf_{x\in[0,\sqrt 2 t -t^{\beta+\epsilon}]}u(t,x)\geq m^*
\text{ and }\lim_{t\to\infty}\sup_{x\geq \sqrt 2 t -t^{\beta-\epsilon}}u(t,x)=0.
$$
Note that $\frac{2-\alpha}{2+\alpha}\to 0$ as $\alpha \uparrow 2$ and $\frac{2-\alpha}{2+\alpha}\to 1$ as $\alpha \downarrow 0$.
\end{remark}
Our final result is a weaker version of Theorem~\ref{thm:heavyupper} with weaker conditions.
\begin{thm} \label{thm:heavyweak}
Suppose that there exist $\alpha\in (0,2)$ and $K<\infty$ such that
\begin{equation*} 
\forall R>0, \exists\, r>R \text{ such that } \int_r^{Kr}\phi (x) dx\geq r^{-\alpha}.
\end{equation*}
Then for any $\beta<\frac{2-\alpha}{2+\alpha}$,  
$$
\liminf_{t\to \infty} \inf_{x\in [0,\sqrt 2 t-t^\beta]}u(t,x)=0 .
$$
\end{thm}

\subsection{Outline of the article}
The proofs of our results are arranged as follows. In Section~\ref{sec:prel}, we prove general results which hold for any $\phi$ satisfying assumption~\eqref{eq:phi_cond}. These results will be used in the proofs in Sections~\ref{sec:light} and~\ref{sec:heavy}. We also give a probabilistic proof of Theorem~\ref{thm:speed}. In Section~\ref{sec:light}, we give a proof of Theorem~\ref{thm:lighttail} and in Section~\ref{sec:heavy} we prove Theorems~\ref{thm:heavylower}, \ref{thm:heavyupper} and \ref{thm:heavyweak}.

\subsection{Notation and main tools} \label{sec:notation}

For $x\in \R$, we shall write $\mathbb P_x$ for the probability measure under which $(B(t),t\geq 0)$ is a Brownian motion started at $x$, and $\mathbb E_x$ for the corresponding expectation.
We shall also write $\mathbb P$ for the probability measure under which, for each $t>0$, $(\xi^t(s),0\leq s \leq t)$ is a Brownian bridge from $0$ to $0$ in time $t$, and $\mathbb E$ for the corresponding expectation.

The main tools used in the proofs will be a form of the Feynman-Kac formula and elementary facts about the Gaussian distribution, which for ease of reference we record here.
\begin{prop}[Feynman-Kac formula] \label{prop:fk}
Suppose for some $T\in (0,\infty)$ and some bounded continuous function $k:[0,T]\times \R\to \R$ that $u:[0,T]\times \R\to \R$ is bounded, continuous on $[0,T]\times \R$ and smooth on $(0,T]\times \R$, and satisfies
$$
\frac{\partial u}{\partial t}=\tfrac{1}{2}\Delta u +k(t,x)u \quad  \forall t\in (0,T], x\in \R.
$$
Then for $t\in [0,T]$, $t'\leq t$, $x\in \R$,
$$
u(t,x)=\Esub{x}{\exp\left(\int_0^{t'} k(t-s,B(s))ds \right) u(t-t',B(t'))}.
$$
\end{prop}
\begin{proof}
This is Theorem 3.3 in Chapter 4 of \cite{durrett1996}.
\end{proof}
Now let $u$ denote the solution to~\eqref{nonlocal_fkpp}.
Using~\eqref{eq:u_bound},
and since $u$ is smooth on $(0,\infty)\times \R$,
 we can apply Proposition~\ref{prop:fk} with $k=1-\phi\ast u$ to obtain that if $0\leq t'< t$ and $x\in \R$,
\begin{equation} \label{feynmankac}
u(t,x)=\Esub{x}{\exp\left(\int_0^{t'} \left(1-\phi \ast u (t-s,B(s))\right)ds \right) u(t-t',B(t'))}.
\end{equation}
By the maximum principle for linear parabolic equations, and by~\eqref{eq:u_bound}, for $t\leq 1$,
$$
\Esub{x}{e^{t(1-e\|u_0\|_\infty)}u_0(B(t))}\leq u(t,x)\leq \Esub{x}{e^t u_0(B(t))}.
$$
Therefore by letting 
$t'\to t$ in~\eqref{feynmankac}, we have
\begin{equation} \label{feynmankac0}
u(t,x)=\Esub{x}{\exp\left(\int_0^{t} \left(1-\phi \ast u (t-s,B(s))\right)ds \right) u_0(B(t))}.
\end{equation}

We shall also need the following elementary facts about the Gaussian distribution.
\begin{lem}
If $Z\sim N(0,1)$, then for $x>0$,
\begin{equation} \label{eq:gaussiantail}
\p{Z>x}\leq e^{-x^2/2}.
\end{equation}
For $x>0$,
\begin{equation} \label{eq:gaussiantail2}
\p{Z>x}\leq \frac{1}{\sqrt{2\pi}}\frac{1}{x}e^{-x^2/2}.
\end{equation}
For $0\leq x\leq y$,
\begin{equation} \label{eq:Btbound}
\frac{y-x}{\sqrt{2\pi t}}e^{-\frac{y^2}{2t}}\leq \psub{0}{B(t)\in [x,y]}\leq \frac{y-x}{\sqrt{2\pi t}}e^{-\frac{x^2}{2t}}.
\end{equation}
\end{lem}
\begin{proof}
The proof of~\eqref{eq:gaussiantail} is by a Chernoff bound. 
A proof of~\eqref{eq:gaussiantail2} is in Lemma 12.9 of \cite{BMbook}.
Finally,~\eqref{eq:Btbound} holds since the density of $B(t)$ is given by $f_{0,t}(z):=\frac{1}{\sqrt{2 \pi t}}e^{-z^2/(2t)}$.
\end{proof}

\section{Preliminary results} \label{sec:prel}

We take $\eta$, $\sigma>0$ and assume from now on that $\phi$ satisfies assumption~\eqref{eq:phi_cond} with this choice of $\eta$ and $\sigma$.
Our first result is a global bound on $u$; this is also proved in Theorem~1.2 of \cite{hamel2014} using PDE methods.
We include a probabilistic proof here as many of the same ideas will be used later on in this section.
\begin{prop} \label{prop:globalbound}
Suppose $u_0\in L^\infty (\R)$ and $u_0\geq 0$ and let $u$ denote the solution of~\eqref{nonlocal_fkpp}.
Then there exists $M=M(\|u_0\|_\infty, \eta, \sigma)<\infty$ such that
$$ 0\leq u(t,x)\leq M \quad \forall t\geq 0, \, x\in \R. $$
\end{prop}
\begin{proof}
Take $\delta>0$ sufficiently small that $e^\delta<4/3$ and $e^{-\sigma^2/(32\delta)}<1/16$, and take $C>\max(\frac{1}{2}\sigma \|u_0\|_\infty ,\frac{2}{\eta \delta}\log 2)$.

For some $t\geq 0$, suppose that $\int_{-\sigma/4}^{\sigma/4}u(t,x+y)dy\leq C$ $\forall x\in \R$.
For $x_0\in \R$ fixed, we consider two cases:
\begin{enumerate}
\item $\int_{-\sigma/2}^{\sigma/2}u(t+s,x_0+y)dy\geq C/2$ for all $s\in [0,\delta]$
\item $\int_{-\sigma/2}^{\sigma/2}u(t+s_0,x_0+y)dy< C/2$ for some $s_0\in [0,\delta]$.
\end{enumerate}
We shall consider each case separately; in each we aim to show that
$\int_{-\sigma/4}^{\sigma/4}u(t+\delta,x_0+y)dy\leq C$.

We begin with case (1).
For $y\in [-\sigma/4,\sigma/4]$,
suppose $|B(s)-(x_0+y)|\leq \sigma/4$ $\forall s\in [0,\delta]$. Then $B(s)\in [x_0-\sigma/2,x_0+\sigma/2]$ $\forall s\in [0,\delta]$ and so
\begin{align*}
\int_0^\delta \phi \ast u(t+\delta-s,B(s))ds&\geq \int_0^\delta \int_{-\sigma}^\sigma \eta u(t+\delta-s,B(s)+y)dy ds\\
&\geq \eta\int_0^\delta \int_{-\sigma/2}^{\sigma/2}u(t+\delta-s,x_0+y)dy ds\\
&\geq \tfrac{1}{2}C\eta \delta,
\end{align*}
where the first inequality holds since $\phi\geq 0$, $u\geq 0$, and $\phi \geq \eta$ a.e.~on $(-\sigma,\sigma)$
and the last inequality holds by our assumption in case (1).
Hence by the Feynman-Kac formula \eqref{feynmankac}, for $y\in [-\sigma/4,\sigma/4]$, since $\phi \ast u\geq 0$,
$$
u(t+\delta,x_0+y)\leq \Esub{x_0+y}{u(t,B(\delta))e^\delta (e^{-\frac{1}{2}C\eta \delta}+\I{\sup_{s\in [0,\delta]}|B(s)-B(0)|\geq \sigma/4})}.
$$
Therefore, by Fubini's Theorem,
\begin{align*}
\int_{-\sigma/4}^{\sigma/4}u(t+\delta,x_0+y)dy
&\leq e^{\delta(1-\frac{1}{2}C\eta)}\Esub{x_0}{\int_{-\sigma/4}^{\sigma/4}u(t,B(\delta)+y)dy}\\
&\qquad +e^\delta \Esub{x_0}{\I{\sup_{s\in [0,\delta]}|B(s)-B(0)|\geq \sigma/4}\int_{-\sigma/4}^{\sigma/4}u(t,B(\delta)+y)dy}\\
&\leq Ce^{\delta(1-\frac{1}{2}C\eta)}
+Ce^\delta \psub{0}{\sup_{s\in [0,\delta]}|B(s)|\geq \sigma/4}\\
&\leq Ce^{\delta(1-\frac{1}{2}C\eta)}
+4Ce^\delta \psub{0}{B(1)\geq \sigma/(4\delta^{1/2})}\\
&\leq Ce^\delta (e^{-\frac{1}{2}C\eta\delta}
+4e^{-\sigma^2/(32\delta)}),
\end{align*}
where the second inequality follows since $\int_{-\sigma/4}^{\sigma/4}u(t,x+y)dy\leq C$ $\forall x\in \R$, the third by the reflection principle and the final inequality by \eqref{eq:gaussiantail}.
By our choice of $\delta$ and $C$ at the start of the proof,
we have $e^\delta (e^{-\frac{1}{2}C\eta\delta}
+4e^{-\sigma^2/(32\delta)})<1$ and hence
$\int_{-\sigma/4}^{\sigma/4}u(t+\delta,x_0+y)dy<C$.

We now consider case (2). 
By the Feynman-Kac formula \eqref{feynmankac} and since $\phi \ast u\geq 0$, and then since $s_0\leq \delta$,
we have that for any $x\in \R$ and $y\in [-\sigma/4,\sigma/4]$,
$$
u(t+s_0,x+y)\leq e^{s_0}\Esub{x}{u(t,B(s_0)+y)}\leq e^{\delta}\Esub{x}{u(t,B(s_0)+y)}.
$$
Then by Fubini's Theorem,
\begin{align} \label{eq:daggersec2}
\int_{-\sigma/4}^{\sigma/4}u(t+s_0,x+y)dy
&\leq e^{\delta}\Esub{x}{\int_{-\sigma/4}^{\sigma/4}u(t,B(s_0)+y)dy}
\leq e^\delta C,
\end{align}
since $\int_{-\sigma/4}^{\sigma/4}u(t,x'+y)dy\leq C$ $\forall x'\in \R$.
By the Feynman-Kac formula \eqref{feynmankac} and since $\phi \ast u\geq 0$, for $y\in [-\sigma/4,\sigma/4]$,
$$
u(t+\delta,x_0+y)\leq e^{\delta-s_0}\Esub{x_0}{u(t+s_0,B(\delta-s_0)+y)}.
$$
Hence by Fubini's Theorem,
\begin{align*}
\int_{-\sigma/4}^{\sigma/4}u(t+\delta,x_0+y)dy
&\leq e^{\delta}\Esub{x_0}{\int_{-\sigma/4}^{\sigma/4}u(t+s_0,B(\delta-s_0)+y)dy}\\
&\leq e^\delta (\tfrac{1}{2}C +e^\delta C\psub{x_0}{|B(\delta-s_0)-x_0|>\sigma/4}),
\end{align*}
by~\eqref{eq:daggersec2} and
since for $x\in [x_0-\sigma/4,x_0+\sigma/4]$, 
$$\int_{-\sigma/4}^{\sigma/4}u(t+s_0,x+y)dy\leq \int_{-\sigma/2}^{\sigma/2}u(t+s_0,x_0+y)dy< C/2$$ by our assumption in case (2).
Therefore, by the reflection principle,
\begin{align*}
\int_{-\sigma/4}^{\sigma/4}u(t+\delta,x_0+y)dy
&\leq C e^\delta \left(\tfrac{1}{2} +2e^\delta \psub{0}{B(1)>\sigma/(4\delta^{1/2})}\right)\\
&\leq C e^\delta (\tfrac{1}{2} +2e ^\delta e^{-\sigma^2/(32\delta)}),
\end{align*}
by \eqref{eq:gaussiantail}.
Again by our choice of $\delta$ and $C$ at the start of the proof, it follows that
$\int_{-\sigma/4}^{\sigma/4}u(t+\delta,x_0+y)dy<C$.

By combining cases (1) and (2) for every $x_0\in \R$ we have that if for some $t\geq 0$, $\int_{-\sigma/4}^{\sigma/4}u(t,x+y)dy\leq C$ $\forall x\in \R$ then $\int_{-\sigma/4}^{\sigma/4}u(t+\delta,x+y)dy\leq C$ $\forall x\in \R$.
Note that $\int_{-\sigma/4}^{\sigma/4}u(0,x+y)dy\leq \frac{1}{2}\sigma\|u_0\|_\infty \leq C$ $\forall x\in \R$ by our choice of $C$.
Therefore $\forall k\in \N$, $x\in \R$,
$\int_{-\sigma/4}^{\sigma/4}u(k\delta,x+y)dy\leq C$.
It follows that for any $t\geq 0$, $x\in \R$, by the Feynman-Kac formula \eqref{feynmankac} and Fubini's Theorem,
$$
\int_{-\sigma/4}^{\sigma/4}u(t,x+y)dy\leq e^{t-\delta \lfloor t/\delta \rfloor}\Esub{x}{\int_{-\sigma/4}^{\sigma/4}u(\delta \lfloor t/\delta \rfloor,B(t-\delta \lfloor t/\delta \rfloor)+y)dy}
\leq e^\delta C.
$$
Hence for $t\geq 1$, $x\in \R$, by \eqref{feynmankac},
\begin{align*}
u(t,x)&\leq e\Esub{x}{u(t-1,B(1))}\\
&\leq e\sum_{k\in\Z}\tfrac{1}{\sqrt{2\pi}}e^{-k^2\sigma^2/32}\int_{-\sigma/4}^{\sigma/4} u(t-1,\tfrac{k\sigma}{4}+y)dy\\
&\leq eCe^\delta \sum_{k\in \Z}\tfrac{1}{\sqrt{2\pi}}e^{-k^2\sigma^2/32}
<\infty.
\end{align*}
Also for $t\leq 1$, $x\in \R$, by~\eqref{feynmankac0},
$$
u(t,x)\leq e^t \Esub{x}{u_0(B(t))}\leq e\|u_0\|_\infty.
$$
The result follows.
\end{proof}
From now on, we take $L<\infty$ and assume that $u_0\in L^\infty(\R)$ with $u_0\geq 0$ and $\|u_0\|_\infty \leq L$, and we shall write $u$ for the solution to~\eqref{nonlocal_fkpp} and $M=M(L,\eta,\sigma)$ as in Proposition~\ref{prop:globalbound}.
We can use the Feynman-Kac formula and the global bound on $u$ to prove a form of uniform continuity.
\begin{lem} \label{lem:abscty}
For $\epsilon>0$ sufficiently small, if $|x-y|<\epsilon^3$ and $t\geq 1$, then $|u(t,x)-u(t,y)|<\epsilon.$
\end{lem}
\begin{proof}
Suppose that $\epsilon<\min(1,(Me(\frac{1}{\sqrt{2\pi}}+M))^{-1})$
and suppose that $t\geq 1$ and $|x-y|<\epsilon^3$.
By the Feynman-Kac formula~\eqref{feynmankac},
for any $z\in \R$,
since $0\leq \phi\ast u\leq M$ by Proposition~\ref{prop:globalbound},
$$
e^{\epsilon^2 (1-M)}\Esub{z}{u(t-\epsilon^2,B(\epsilon^2))}
\leq u(t,z)
\leq e^{\epsilon^2}\Esub{z}{u(t-\epsilon^2,B(\epsilon^2))}.
$$
Therefore
\begin{align}\label{eq:pdfint}
u(t,x)-u(t,y)
&\leq e^{\epsilon^2}\Esub{x}{u(t-\epsilon^2,B(\epsilon^2))}
-e^{\epsilon^2(1-M)}\Esub{y}{u(t-\epsilon^2,B(\epsilon^2))} \notag\\
&=e^{\epsilon^2}\left(\Esub{x}{u(t-\epsilon^2,B(\epsilon^2))}-\Esub{y}{u(t-\epsilon^2,B(\epsilon^2))}\right)\notag\\
&\hspace{2cm}+(e^{\epsilon^2}-e^{\epsilon^2(1-M)})\Esub{y}{u(t-\epsilon^2,B(\epsilon^2))}\notag\\
&\leq e^{\epsilon^2}\left(\Esub{x}{u(t-\epsilon^2,B(\epsilon^2))}-\Esub{y}{u(t-\epsilon^2,B(\epsilon^2))}\right)+(e^{\epsilon^2}-e^{\epsilon^2(1-M)})M,
\end{align}
by Proposition~\ref{prop:globalbound}.
Write $f_{\mu,\sigma^2}$ for the density of the Gaussian distribution with mean $\mu$ and variance $\sigma^2$.
If $x\leq y$, then by Proposition~\ref{prop:globalbound},
\begin{align*}
&\Esub{x}{u(t-\epsilon^2,B(\epsilon^2))}-\Esub{y}{u(t-\epsilon^2,B(\epsilon^2))}\\
&\qquad =
\int_{-\infty}^\infty u(t-\epsilon^2,z)(f_{x,\epsilon^2}(z)- f_{y,\epsilon^2}(z))dz\\
&\qquad\leq M\int_{-\infty}^{(x+y)/2}(f_{x,\epsilon^2}(z)- f_{y,\epsilon^2}(z))dz\\
&\qquad= M \left(\psub{x}{B(\epsilon^2)\leq \tfrac{1}{2}(x+y)}-\psub{y}{B(\epsilon^2)\leq \tfrac{1}{2}(x+y)}\right)\\
&\qquad= M \left(\psub{0}{B(\epsilon^2)\leq \tfrac{1}{2}(y-x)}-\psub{0}{B(\epsilon^2)\leq \tfrac{1}{2}(x-y)}\right)\\
&\qquad=M \psub{0}{|B(\epsilon^2)|\leq \tfrac{1}{2}|y-x|}.
\end{align*}
Similarly, if $y\leq x$,
\begin{align*}
&\Esub{x}{u(t-\epsilon^2,B(\epsilon^2))}-\Esub{y}{u(t-\epsilon^2,B(\epsilon^2))}\\
&\qquad\leq 
M \left(\psub{x}{B(\epsilon^2)\geq \tfrac{1}{2}(x+y)}-\psub{y}{B(\epsilon^2)\geq \tfrac{1}{2}(x+y)}\right)\\
&\qquad =M \psub{0}{|B(\epsilon^2)|\leq \tfrac{1}{2}|y-x|}.
\end{align*}
Substituting into \eqref{eq:pdfint},
\begin{align*} 
u(t,x)-u(t,y)
&\leq 
Me^{\epsilon^2}\left(\psub{0}{|B(\epsilon^2)|\leq \tfrac{1}{2}|y-x|}
+1-e^{-M\epsilon^2}\right)\\
&\leq M e \left(\frac{|x-y|}{\sqrt{2\pi \epsilon^2}}+M\epsilon^2\right),
\end{align*}
by \eqref{eq:Btbound} and since $1-e^{-r}\leq r$ for $r\geq 0$.
Since $|x-y|\leq \epsilon^3$, it follows that
$$u(t,x)-u(t,y)\leq Me\left(\tfrac{1}{\sqrt{2\pi}}+M\right)\epsilon^2 \leq \epsilon$$ by our choice of $\epsilon$ at the start of the proof.
By the same argument, $u(t,y)-u(t,x)\leq \epsilon$, and the result follows.
\end{proof}
We now show that if $\delta>0$, and $u$ is small, then $u$ grows exponentially until there is some $x$ nearby with $u(t,x)\geq 1-\delta$.
The proof of this result uses ideas from the proof of Lemma~5.4 in \cite{addario2015}.
\begin{lem} \label{lem:growu}
For $\delta\in (0,1)$, there exist $C=C(\delta)$, $R=R(\delta)$ and $z_0=z_0(\delta)$ such that for $z\in (0,z_0)$,
if $t\geq 1$ and $u(t,x)>z$ then there exist $s\in [t,t+C\log (1/z)]$
and $y\in [x-R,x+R]$ such that $u(s,y)\geq 1-\delta.$
\end{lem}
\begin{proof}
Take $R>0$ sufficiently large that
$R^2>\pi^2/\delta$ and
$\int_{|r|\geq R/2}\phi(r)dr<\delta/(2M)$.
Then take $C>8(\delta - \pi^2 R^{-2})^{-1}$.
The proof is divided into the following two cases:
\begin{enumerate}
\item For each $s\in [t,t+C\log (1/z)]$, $y\in [x-\tfrac{1}{2}R,x+\tfrac{1}{2}R]$, we have $\phi \ast u (s,y)<1-\tfrac{1}{2}\delta$.
\item There exist $s_0\in [t,t+C\log (1/z)]$ and $y_0\in [x-\tfrac{1}{2}R,x+\tfrac{1}{2}R]$ such that $\phi \ast u (s_0,y_0)\geq 1-\tfrac{1}{2}\delta$.
\end{enumerate}
We shall begin with case (1).
Since $u(t,x)>z$, by Lemma~\ref{lem:abscty} with $\epsilon=\frac{1}{2}z$ we have that if $z$ is sufficiently small then 
\begin{equation} \label{eq:z3_mass}
u(t,y)>\tfrac{1}{2}z \quad \forall y\in [x-\tfrac{1}{8} z^3,x+\tfrac{1}{8} z^3].
\end{equation}
Also, if $|B(s)-x|\leq \tfrac{1}{2}R$ $\forall s\in [0,C\log (1/z)]$ then 
\begin{equation} \label{eq:int_est_phiu}
\int_0^{C\log(1/z)}\phi\ast u(t+C\log(1/z)-s,B(s))ds
<C\log (1/z)(1-\tfrac{1}{2}\delta)
\end{equation}
by our assumption in case (1).
Therefore by the Feynman-Kac formula \eqref{feynmankac}, 
\begin{align*}
&u(t+C\log (1/z),x)\\
&=\Esub{x}{\exp\left(\int_0^{C\log(1/z)}(1-\phi\ast u(t+C\log(1/z)-s,B(s)))ds \right) u(t,B(C\log(1/z)))}\\
&\geq \tfrac{1}{2}z e^{C\log (1/z)\delta /2}
\psub{x}{|B(C\log (1/z))-x|\leq \tfrac{1}{8}z^3, |B(s)-x|\leq \tfrac{1}{2}R \, \,\forall s \leq C\log (1/z)}\\
&= \tfrac{1}{2}z e^{C\log (1/z)\delta /2}
\psub{0}{|B(4CR^{-2}\log (1/z))|\leq \tfrac{1}{4}z^3 R^{-1}, |B(s)|\leq 1 \, \,\forall s \leq 4CR^{-2}\log (1/z)},
\end{align*}
where the first inequality follows by~\eqref{eq:z3_mass} and~\eqref{eq:int_est_phiu} and since $u\geq 0$, and the last line follows by Brownian scaling. 

By Lemma 5 in \cite{roberts2015} (which gives a convenient statement of this well known estimate), for $T>0$, there exists a constant $c_T>0$ such that for any $t\geq T$ and $0\leq x_0\leq 1$,
\begin{align} \label{eq:BMtube}
\psub{0}{|B(s)|\leq 1 \,\, \forall s \leq t, \,|B(t)|\leq x_0}
&\geq c_T e^{-\pi^2 t/8}\int_{-x_0}^{x_0}\cos (\pi \nu/2)d\nu \notag \\
&\geq c_T e^{-\pi^2 t/8}\min (x_0,\tfrac{2}{3}), 
\end{align}
since $\cos(\pi \nu/2)\geq 1/2$ for $|\nu|\leq 2/3$.
Hence for $z$ sufficiently small
that $4CR^{-2}\log(1/z)\geq 1$ and $\frac{1}{4}z^3 R^{-1}\leq \frac{2}{3}$,
\begin{align*}
u(t+C\log (1/z),x)
&\geq\tfrac{1}{2}z e^{C\log (1/z)\delta /2}
c_1 e^{-\pi^2 4CR^{-2}\log (1/z)/8}\tfrac{1}{4}z^3 R^{-1}\\
&= \tfrac{1}{8}c_1 R^{-1}z^{4+\pi^2 CR^{-2}/2-C\delta/2}. 
\end{align*}
Recall that $R^2>\pi^2/\delta$ and $C>8(\delta - \pi^2 R^{-2})^{-1}$, so
$a:= 4+C(\pi^2 R^{-2}-\delta)/2<0$.
It follows that as long as $z\leq (8c_1^{-1}R(1-\delta))^{1/a}$, then
$u(t+C\log (1/z),x)\geq 1-\delta$.

We now consider case (2). Recall that $\int_{|r|\geq R/2}\phi(r)dr<\delta/(2M)$.
By Proposition~\ref{prop:globalbound} and the assumption that we are in case (2), 
\begin{equation} \label{eq:phi_far}
M\int_{|r|\geq R/2}\phi(r)dr+\int_{|r|\leq R/2}\phi(r)u(s_0,y_0-r)dr
\geq \phi \ast u (s_0,y_0)\geq 1-\tfrac{1}{2}\delta,
\end{equation}
and therefore
$$
\int_{|r|\leq R/2}\phi(r)u(s_0,y_0-r)dr
\geq  1-\delta.
$$
Since $\int_{-\infty}^\infty \phi(r)dr=1$ and $\phi\geq 0$,
it follows that there exists $r\in [-\frac{1}{2}R,\frac{1}{2}R]$ such that $u(s_0,y_0-r)\geq 1-\delta$. 
Since $s_0\in [t,t+C\log (1/z)]$ and $|(y_0-r)-x|\leq R$, this completes the proof.
\end{proof}
Our next lemma shows that for any $\epsilon>0$, $u$ spreads at speed at least $\sqrt 2 -\epsilon$.
\begin{lem} \label{lem:travelu}
For $0\leq c<\sqrt 2$, there exists $m^*=m^*(c)\in (0,1/2)$ and $t^*=t^*(c)<\infty$ such that for $T\geq t^*$ and $t\geq 1$, if $u(t,x)\geq m^*$ and $|x'-x|\leq cT$, then
$u(t+T,x')\geq m^*$. 
\end{lem}
\begin{proof}
We shall use the following estimate on the probability that a Brownian motion stays inside a tilted tube.

By Girsanov's Theorem, for $b\in \R$, $R_0>0$, $r\in [0,R_0]$ and $t>0$,
\begin{align*}
&\psub{0}{|B(s)-bs|\leq R_0\, \forall s\leq t, |B(t)-bt|\leq r}\\
&\hspace{0.5cm}= \Esub{0}{e^{-b(B(t)+bt)+\frac{1}{2}b^2 t}\I{|B(s)|\leq R_0 \, \forall s\leq t, |B(t)|\leq r}}\\
&\hspace{0.5cm}\geq e^{-\frac{1}{2}b^2 t -bR_0}\psub{0}{|B(s)|\leq R_0\, \forall s\leq t, |B(t)|\leq r}\\
&\hspace{0.5cm}= e^{-\frac{1}{2}b^2 t -bR_0}\psub{0}{|B(s)|\leq 1\, \forall s\leq t R_0^{-2}, |B(t R_0^{-2})|\leq r R_0^{-1}}
\end{align*}
by Brownian scaling.
Hence if also $t\geq 1$ and $|b|\leq \sqrt 2$, then by \eqref{eq:BMtube},
\begin{align} \label{eq:tilttube}
\psub{0}{|B(s)-bs|\leq R_0\, \forall s\leq t, |B(t)-bt|\leq r}
&\geq e^{-\frac{1}{2}b^2 t-\sqrt 2 R_0}c_{R_0^{-2}}  e^{-\frac{1}{8}\pi^2 R_0^{-2}t}\tfrac{2}{3}\tfrac{r}{R_0} \notag \\
&= \delta_{R_0} r e^{-\frac{1}{2}b^2 t-\frac{1}{8}\pi^2 R_0^{-2}t},
\end{align}
where $\delta_{R_0}:= e^{-\sqrt 2 R_0} \tfrac{2}{3}R_0^{-1}c_{R_0^{-2}}>0$.

We now define some constants (the reasons for the conditions imposed should become clear in the course of the proof).
Take $m_0>0$ sufficiently small and $R>1$ sufficiently large that
\begin{equation} \label{eq:m0cond}
1-m_0-\tfrac{1}{2}c^2-\tfrac{1}{8}\pi^2 (R-1)^{-2}>0.
\end{equation}
By Lemma~\ref{lem:abscty}, there exists $\epsilon_0\in (0,1/2)$ such that if $s\geq 1$ and $|y-y'|\leq \epsilon_0$ then $|u(s,y)-u(s,y')|\leq \tfrac{1}{4}m_0$. 
Take $R'>0$ sufficiently large that $\int_{|r|\geq R'}\phi(r)dr<m_0/(2M)$.
Take $m^*>0$ sufficiently small that
\begin{equation} \label{eq:m*def1}
 m^*<\tfrac{1}{4}e^{-3M}m_0 \tfrac{2\epsilon_0}{\sqrt{6\pi}}\exp(-\tfrac{1}{2}(3c+R+R'+1)^2)
\end{equation}
and also
\begin{equation} \label{eq:m*def2}
m^*<e^{-2M}\tfrac{1}{8}m_0 \psub{R+R'+c+1}{|B(1)|\leq \epsilon_0}
\delta_{R-1}\psub{0}{|B(1)-c|\leq 1/2}.
\end{equation}
Again by Lemma~\ref{lem:abscty}, there exists $\epsilon\in (0,1/2)$ such that if $s\geq 1$ and $|y-y'|\leq \epsilon$ then $|u(s,y)-u(s,y')|\leq \frac{1}{2}m^*$. 
Finally take $t^*>3$ sufficiently large that 
\begin{equation} \label{eq:starcond}
\tfrac{1}{4}\epsilon \delta_{R-1} e^{(1-m_0-\frac{1}{2}c^2 -\frac{1}{8}\pi^2 (R-1)^{-2})(t^*-1)}
>e^M (\psub{0}{|B(1)-c|\leq \epsilon/2})^{-1}.
\end{equation}

Take $t\geq 1$, $T\geq t^*$ and $x,x'\in \R$ with $|x-x'|\leq cT$, and suppose $u(t,x)\geq m^*$.
Suppose $x'\geq x$ (the proof for $x'\leq x$ is the same), and let $a=(x'-x)/T \leq c$.
The proof is divided into three cases.
\begin{enumerate}
\item For each $(s,y)$ with $s\in [0,T-1]$, $|y-(x+as)|\leq R$, we have $\phi \ast u (t+s,y)<m_0$.
\item There exist $(s_0,y_0)$ with $s_0\in [T-3,T-1]$ and $|y_0-(x+as_0)|\leq R$ such that $\phi \ast u (t+s_0,y_0)\geq m_0$.
\item There exist $(s_0,y_0)$ with $s_0\in [0,T-3]$ and $|y_0-(x+as_0)|\leq R$ such that $\phi \ast u (t+s_0,y_0)\geq m_0$,
and such that for any $(s,y)$ with $s\in [s_0+1,T-1]$, 
$|y-(x+as)|\leq R$, we have $\phi \ast u (t+s,y)<m_0$.
\end{enumerate}
We shall treat each case separately; in each we aim to show that $u(t+T,x')\geq m^*$.

We begin with case (1).
By our choice of $\epsilon$, we have that
$u(t,y)\geq \frac{1}{2}m^*$ $\forall y\in [x-\epsilon,x+\epsilon]$.
Also for $s\in [0,T-1]$, if
$|B(s)-(x+a(T-1-s))|\leq R$ then $\phi \ast u(t+T-1-s,B(s))<m_0$
by our assumption in case (1).
Hence for $y\in [x-\epsilon/2,x+\epsilon/2]$, by the Feynman-Kac formula~\eqref{feynmankac} and since $u\geq 0$,
\begin{align*}
&u(t+T-1,y+a(T-1))\\
&\geq \tfrac{1}{2}m^* e^{(1-m_0)(T-1)}\psub{y+a(T-1)}{|B(T-1)-x|\leq \epsilon, \,|B(s)-(x+a(T-1-s))|\leq R\, \forall s\leq T-1}\\
&\geq \tfrac{1}{2}m^* e^{(1-m_0)(T-1)}\psub{0}{|B(T-1)+a(T-1)|\leq \epsilon/2, \,|B(s)+as|\leq R-1\, \forall s\leq T-1}
\end{align*} 
since $|y-x|\leq \epsilon/2$.
Therefore by the estimate in \eqref{eq:tilttube},
since $T\geq t^*>3$, $R>1$ and $|a|\leq \sqrt 2$, for $y\in [x-\epsilon/2,x+\epsilon/2]$,
\begin{align} \label{eq:five_star_2}
u(t+T-1,y+a(T-1))
&\geq \tfrac{1}{2}m^* e^{(1-m_0)(T-1)}\delta_{R-1}\tfrac{1}{2}\epsilon e^{(-\frac{1}{2}a^2 -\frac{1}{8}\pi^2 (R-1)^{-2})(T-1)}\notag \\
&\geq \tfrac{1}{4}\epsilon \delta_{R-1} m^* e^{(1-m_0-\frac{1}{2}c^2 -\frac{1}{8}\pi^2 (R-1)^{-2})(t^*-1)} \notag\\
&\geq m^* e^M (\psub{0}{|B(1)-c|\leq \epsilon/2})^{-1},
\end{align} 
where the second inequality follows since $T\geq t^*>3$, $0\leq a\leq c$ and by the choice of constants in~\eqref{eq:m0cond},
and the third inequality follows
by our choice of constants in~\eqref{eq:starcond}.
By Proposition~\ref{prop:globalbound}, we have $\phi\ast u \leq M$;
it follows by the Feynman-Kac formula~\eqref{feynmankac} and since $x'=x+aT$ that
\begin{align*}
u(t+T,x')
&\geq \inf_{|y-(x+a(T-1))|\leq \epsilon/2} u(t+T-1,y) e^{-M}
\psub{x+aT}{|B(1)-(x+a(T-1))|\leq \epsilon/2}\\
&\geq m^* e^M (\psub{0}{|B(1)-c|\leq \epsilon/2})^{-1}e^{-M}
\psub{0}{|B(1)+a|\leq \epsilon/2}\\
&\geq m^*,
\end{align*}
where the second inequality follows by~\eqref{eq:five_star_2} and the third inequality
since $0\leq a \leq c$.

We now move on to case (2). 
Recall that $M\int_{|r|\geq R'}\phi(r)dr<\tfrac{1}{2}m_0$
by our choice of $R'$,
so by the same argument as in \eqref{eq:phi_far}, there exists $y_1\in [y_0-R',y_0+R']$ such that $u(t+s_0,y_1)\geq \frac{1}{2}m_0$.
Then by our choice of $\epsilon_0$, since $u(t+s_0,y_1)\geq \frac{1}{2}m_0$, we have  $u(t+s_0,y)\geq \frac{1}{4}m_0$ $\forall y \in [y_1-\epsilon_0,y_1+\epsilon_0]$.
Hence by the Feynman-Kac formula \eqref{feynmankac} and Proposition~\ref{prop:globalbound}, and since $T-s_0\leq 3$ and $x'=x+aT$,
\begin{align*}
u(t+T,x')&\geq
e^{-3M}\tfrac{1}{4}m_0 \psub{x+aT}{|B(T-s_0)-y_1|\leq \epsilon_0}\\
&\geq
e^{-3M}\tfrac{1}{4}m_0 \psub{x+aT}{|B(T-s_0)-(x+as_0-R-R')|\leq \epsilon_0}\\
&\geq
e^{-3M}\tfrac{1}{4}m_0 \psub{0}{|B(T-s_0)+(3a+R+R')|\leq \epsilon_0}\\
&\geq
e^{-3M}\tfrac{1}{4}m_0 \tfrac{2\epsilon_0}{\sqrt{6\pi}}\exp(-\tfrac{1}{2}(3a+R+R'+1)^2),
\end{align*}
where the second inequality follows since $|y_1-(x+as_0)|\leq R+R'$ and the third and final inequalities follow from $T-s_0\in [1,3]$ and \eqref{eq:Btbound}.
By our choice of $m^*$ in \eqref{eq:m*def1}, and since $0\leq a \leq c$, we have that $u(t+T,x')\geq m^*$.

Finally, we consider case (3).
By the same argument as in case (2), there exists $y_1\in [y_0-R',y_0+R']$ such that $u(t+s_0,y)\geq \frac{1}{4}m_0$ $\forall y \in [y_1-\epsilon_0,y_1+\epsilon_0]$.
Hence for $y\in [x-1,x+1]$, by the Feynman-Kac formula~\eqref{feynmankac} and Proposition~\ref{prop:globalbound},
\begin{align} \label{eq:dagger_est}
u(t+s_0+1,y+a(s_0+1))
&\geq e^{-M}\tfrac{1}{4}m_0 \psub{y+a(s_0+1)}{|B(1)-y_1|\leq \epsilon_0} \notag\\
&\geq e^{-M}\tfrac{1}{4}m_0 \psub{R+R'+c+1}{|B(1)|\leq \epsilon_0}
\end{align}
since $|y_1-y_0|\leq R'$, $|y_0-(x+as_0)|\leq R$, $|x-y|\leq 1$ and $0\leq a \leq c$ 
so $|y_1-(y+a(s_0+1))|\leq R+R'+c+1$.
By the choice of $s_0$ and the assumption of case (3),
for $s\in [0,T-s_0-2]$, if $|B(s)-(x+a(T-1-s))|\leq R$, we have $\phi \ast u (t+T-1-s,B(s))<m_0$.
Therefore by~\eqref{feynmankac} again, for $y\in [x-1/2,x+1/2]$,
\begin{align*}
&u(t+T-1, y+a(T-1))\\
&\geq \inf_{y'\in [x-1,x+1]}u(t+s_0+1,y'+a(s_0+1))
e^{(1-m_0)(T-s_0-2)}\\
&\qquad\mathbb P_{y+a(T-1)}(|B(s)-(y+a(T-1-s))|\leq R-1\, \forall s \leq T-s_0-2,\\
&\hspace{5cm}|B(T-s_0-2)-(y+a(s_0+1))|\leq 1/2).
\end{align*}
Therefore by \eqref{eq:tilttube} and \eqref{eq:dagger_est}, and since $T-s_0-2\geq 1$,
\begin{align} \label{eq:(A)2}
&u(t+T-1, y+a(T-1))\notag\\
&\geq e^{-M}\tfrac{1}{4}m_0 \psub{R+R'+c+1}{|B(1)|\leq \epsilon_0}
e^{(1-m_0)(T-s_0-2)}\delta_{R-1}\tfrac{1}{2}e^{-(\frac{1}{2}a^2+\frac{1}{8}\pi^2 (R-1)^{-2})(T-s_0-2)}\notag\\
&\geq e^{-M}\tfrac{1}{8}m_0 \psub{R+R'+c+1}{|B(1)|\leq \epsilon_0}
\delta_{R-1},
\end{align}
since $1-m_0-\frac{1}{2}c^2-\frac{1}{8}\pi^2 (R-1)^{-2}>0$ (by our choice of constants in~\eqref{eq:m0cond}) and $0\leq a \leq c$.
Finally, by \eqref{feynmankac} and Proposition~\ref{prop:globalbound},
and since $x'=x+aT$,
\begin{align*}
u(t+T,x')
&\geq e^{-M}\inf_{|y-x|\leq 1/2}u(t+T-1,y+a(T-1))\\
&\hspace{3cm}\psub{x+aT}{|B(1)-(x+a(T-1))|\leq 1/2}\\
&\geq e^{-2M}\tfrac{1}{8}m_0 \psub{R+R'+c+1}{|B(1)|\leq \epsilon_0}
\delta_{R-1}\psub{0}{|B(1)-c|\leq 1/2},
\end{align*}
by~\eqref{eq:(A)2} and since $0\leq a \leq c$.
By our choice of $m^*$ in \eqref{eq:m*def2}, we have that $u(t+T,x')\geq m^*$.
This completes the proof.
\end{proof}
Recall that $\|u_0\|_\infty \leq L$.
From now on, we shall assume that $u_0 (x)=0$ $\forall x\geq L$
and $u_0 \not \equiv 0$.
We can now prove Theorem~\ref{thm:speed} with $m^*=m^*(1)$.

\begin{proof}[Proof of Theorem~\ref{thm:speed}]
We begin by proving an upper bound on $u$.
For $y\geq 0$ and $t\geq L/(\sqrt 2-1)$, by the Feynman-Kac formula \eqref{feynmankac0},
\begin{align} \label{uabovesqrt2}
u(t,\sqrt 2 t +y)
&\leq e^t \Esub{\sqrt 2 t  +y}{u_0(B(t))} \notag\\
&\leq L e^t \psub{\sqrt 2 t +y}{B(t)\leq L}\notag\\
&\leq Le^t \frac{1}{\sqrt{2\pi t}}\exp\left(-\frac{1}{2t}(\sqrt 2 t +y -L)^2\right)\notag\\
&\leq \frac{L}{\sqrt{2\pi t}}e^{\sqrt 2 L},
\end{align}
where the 
second line follows since $\|u_0\|_\infty \leq L$ and $u_0(x)=0$ for $x\geq L$ and the
third line follows by~\eqref{eq:gaussiantail2}
and since $(\sqrt 2 t -L)t^{-1/2}\geq t^{1/2}$.
It follows that $\lim_{t\to\infty}\sup_{x\geq \sqrt 2 t}u(t,x)=0$.

It remains to prove a lower bound.
Fix $\epsilon >0$.
Note first that by \eqref{feynmankac0} and Proposition~\ref{prop:globalbound},
\begin{equation} \label{eq:u>0}
u(1,0)\geq e^{1-M}\Esub{0}{u_0(B(1))}>0
\end{equation}
since $u_0\geq 0$ and $u_0\not \equiv 0$.
Let $C=C(1/2)$, $R=R(1/2)$ and $z_0=z_0(1/2)$ as defined in Lemma~\ref{lem:growu}.
Then by Lemma~\ref{lem:growu}, there exist $s_0\in[0,C\log (1/\min(z_0,u(1,0)))]$ and $y_0\in [-R,R]$ such that $u(1+s_0,y_0)\geq 1/2$.

By Lemma~\ref{lem:travelu}, for $t\geq t^*(\sqrt 2 -\epsilon)$, 
if $|x-y_0|\leq (\sqrt 2-\epsilon)t$ then 
$u(1+s_0+t,x)\geq m^*(\sqrt 2-\epsilon)$.
By Lemma~\ref{lem:growu} again, if $|x- y_0|\leq (\sqrt 2-\epsilon)t$ then there exists $s_1(x)\in [0,C\log (1/\min(z_0,m^*(\sqrt 2-\epsilon)))]$ and $y_1(x)\in [-R,R]$
such that $u(1+s_0+t+s_1(x),x+y_1(x))\geq 1/2$.
Then by Lemma~\ref{lem:travelu}, it follows that for any $s_2\geq \max(t^*(1),2R)$ and $y\in [x-R,x+R]$,
$u(1+s_0+t+s_1(x)+s_2,y)\geq m^*(1)$.

Let $A=1+C\log (1/\min(z_0,u(1,0)))+C\log (1/\min(z_0,m^*(\sqrt 2-\epsilon)))+\max(t^*(1),2R)$.
We now have that for $t\geq t^*(\sqrt 2 -\epsilon)$, for any $s\geq t+A$,
for any $x\in [-(\sqrt 2-\epsilon)t,(\sqrt 2-\epsilon)t]$, $u(s,x)\geq m^*(1)$.
Therefore, for $s\geq t^*(\sqrt 2 -\epsilon)+A$,
$\inf_{x\in [0, (\sqrt 2 -\epsilon)s-\sqrt 2 A]}u(s,x)\geq m^*(1)$.
Since $\epsilon>0$ was arbitrary, the result follows.
\end{proof}
The final lemma of this section will help us determine regions in which $\phi\ast u(t,y) \lesssim t^{-1}$, and therefore Brownian paths for which $\int_0^{t-1} \phi\ast u(t-s,B(s))ds =\mathcal O(\log t)$.
\begin{lem} \label{lem:t-1u}
For $y\geq 0$ and $t\geq \max(L/(\sqrt 2 -1),1)$,
$$
u(t,\sqrt 2 t +\tfrac{1}{2\sqrt 2}\log t +y)\leq L e^{2 L} t^{-1}.
$$
\end{lem}
\begin{proof}
By the Feynman-Kac formula \eqref{feynmankac0}, for $y\geq 0$ and $t\geq \max(L/(\sqrt 2 -1),1)$,
\begin{align*}
u(t,\sqrt 2 t +\tfrac{1}{2\sqrt 2}\log t +y)
&\leq e^t \Esub{\sqrt 2 t +\frac{1}{2\sqrt 2}\log t +y}{u_0(B(t))} \notag\\
&\leq L e^t \psub{\sqrt 2 t +\frac{1}{2\sqrt 2}\log t +y}{B(t)\leq L} \notag\\
&\leq Le^t \frac{1}{\sqrt{2\pi t}}\exp\left(-\frac{1}{2t}(\sqrt 2 t +\tfrac{1}{2\sqrt 2}\log t -L)^2\right) \notag\\
&\leq \frac{L}{\sqrt{2\pi}}e^{\sqrt 2 L+\frac{L\log t}{2\sqrt 2t}}t^{-1}\notag \\
&\leq L e^{2 L} t^{-1},
\end{align*}
where the second line follows since $\|u_0\|_\infty \leq L$ and $u_0(x)=0$ for $x\geq L$, the third line follows by \eqref{eq:gaussiantail2} and since $\sqrt{2}t\geq L$, $y\geq 0$ and $(\sqrt 2 t +\frac{1}{2\sqrt 2}\log t-L)t^{-1/2}\geq t^{1/2}$, and the last line follows
since $t^{-1}\log t\leq e^{-1}$.
\end{proof}

\section{Proof of Theorem~\ref{thm:lighttail}} \label{sec:light}

In this section we shall suppose that there exists $\alpha>2$ such that for $r_0>0$ sufficiently large,  
\begin{equation} \label{eq:lighttail}
\forall r\geq r_0, \,\int_r^\infty \phi(x)dx \leq r^{-\alpha}.
\end{equation}
As in Section~\ref{sec:prel}, we assume that $\phi$ satisfies assumption~\eqref{eq:phi_cond} with our choice of $\eta$ and $\sigma$.
Also we suppose that $u_0\in L^\infty (\R)$ with $u_0\geq 0$, $u_0 \not \equiv 0$, $\|u_0\|_\infty \leq L$ and $u_0 (x)=0$ $\forall x\geq L$, and let $u$ denote the solution of~\eqref{nonlocal_fkpp}.

We shall use the following pair of lemmas from~\cite{bramson1983}.
The first is an application of the reflection principle.
Recall from Section~\ref{sec:notation} that under $\mathbb P$, $(\xi^t(s),0\leq s \leq t)$ is a Brownian bridge from $0$ to $0$ of length $t$.
\begin{lem} [Lemma 2.2 in~\cite{bramson1983}] \label{lem:bram1}
For $y_1,y_2>0$,
$$
\p{\xi^t(s)\geq -\tfrac{s}{t}y_1-\tfrac{t-s}{t}y_2 \, \,\forall s\in [0,t]}=1-e^{-2y_1 y_2 /t}.
$$
\end{lem} 
The second lemma is an example of a phenomenon known as entropic repulsion, and is proved using a Girsanov transform.
\begin{lem}[Simplified version of Lemma 6.1 in \cite{bramson1983}]
\label{lem:bram2}
For $z\in \R$, $\delta \in (0,1/2)$ and $A>0$ fixed, 
for $\epsilon>0$, there exists $r_\epsilon<\infty$ such that $\forall r>r_\epsilon$ and $t>3r$,
$$
\left|\frac{\p{\xi^t(s)>z+\min(As^\delta, A(t-s)^\delta)\, \forall s\in [r,t-r]}}{\p{\xi^t(s)>z-\min(As^\delta, A(t-s)^\delta)\, \forall s\in [r,t-r]}}-1 \right| <\epsilon.
$$
\end{lem}

The first step in the proof of Theorem~\ref{thm:lighttail} is the following result.
Take $m^*(1)>0$ as defined in Lemma~\ref{lem:travelu}.

\begin{prop} \label{prop:logdelay}
There exist $K<\infty$ and $T<\infty$ such that for $t\geq T$, 
$u(t,x)\geq m^*(1)$ $\forall x\in [0,\sqrt 2 t-K\log t]$.
\end{prop}
\begin{proof}
Recall from~\eqref{eq:u>0} in the proof of Theorem~\ref{thm:speed} that since $u_0\not \equiv 0$, $u_0\geq 0$ and $u_0\in L^\infty (\R)$ we have $u(1,0)>0$.
Now take $\epsilon\in (0,\frac{1}{2}u(1,0))$ sufficiently small that Lemma~\ref{lem:abscty} holds for this choice of $\epsilon$.
It follows that for $|x|\leq \epsilon^3$, we have $u(1,x)\geq \epsilon$.

Therefore, for $t\geq e$,
for $x\in [\frac{5}{4}\sqrt 2 \log t +1,\frac{5}{4}\sqrt 2 \log t +2]$,
by the Feynman-Kac formula~\eqref{feynmankac} and Proposition~\ref{prop:globalbound},
\begin{align} \label{eq:someulog}
u(\log t, x)
&\geq e^{(\log t-1)(1-M)}\epsilon \psub{x}{|B(\log t-1)|\leq \epsilon^3}\notag \\
&\geq \epsilon e^{M-1}t^{-M+1}\frac{2\epsilon^3}{\sqrt{2\pi \log t}}
\exp(-(\epsilon^3+2+\tfrac{5}{4}\sqrt 2 \log t)^2/(2(\log t-1))) \notag\\
&\geq \frac{2\epsilon^4}{\sqrt{2\pi }} e^{M-1}(\log t)^{-1/2}t^{-M+1}
\exp(-\tfrac{25}{16}\log t -o(\log t)) \notag\\
&\geq t^{-M-1}
\end{align} 
for $t$ sufficiently large, where the second line follows by~\eqref{eq:Btbound}.

Take $\delta\in (1/\alpha,1/2)$ and $t$ large.
Then for $s\in [\max(r_0^{1/\delta},L/(\sqrt 2-1)),(t-\log t)/2]$, if 
$y\geq \sqrt 2 (s+\log t)+\frac{1}{2\sqrt 2}\log t+s^\delta$,
we have
\begin{align} \label{eq:phiulow1}
\phi\ast u(s+\log t,y)&=\int_{-\infty}^\infty \phi (x)u(s+\log t,y-x)dx \notag \\
&\leq M\int_{s^\delta}^\infty \phi (x)dx+\sup_{z\geq \sqrt 2 (s+\log t) +\frac{1}{2\sqrt2}\log t}u(s+\log t,z)\int_{-\infty}^\infty \phi (x) dx \notag \\
&\leq Ms^{-\alpha \delta}+L e^{2 L}(s+\log t)^{-1},
\end{align}
where the second line follows by Proposition~\ref{prop:globalbound} and since $\phi \geq 0$ and the last line holds for $s\geq \max(r_0^{1/\delta},L/(\sqrt 2-1))$ by \eqref{eq:lighttail} and Lemma \ref{lem:t-1u} and since $s+\log t\leq \tfrac{1}{2}(t+\log t)\leq t$ and $\int_{-\infty}^\infty \phi (x) dx=1$.
Similarly, for $t$ sufficiently large, if $s\in [(t-\log t)/2,t-\log t-r_0^{1/\delta}]$ and 
$y\geq \sqrt 2 (s+\log t)+\frac{1}{2\sqrt 2}\log t+(t-\log t-s)^\delta$,
\begin{align} \label{eq:phiulow2}
\phi\ast u(s+\log t,y)
&\leq M(t-\log t-s)^{-\alpha \delta}+L e^{2 L}(s+\log t)^{-1}.
\end{align}
For $s\in[0,t-\log t]$, let 
\begin{equation} \label{eq:fdef}
f(s)=\sqrt 2 (t-s)+\tfrac{1}{2\sqrt 2}\log t+\min(s^\delta,(t-\log t-s)^\delta).
\end{equation}
Note that if $t-s=s'+\log t$ then $f(s)=\sqrt 2 (s'+\log t)+\frac{1}{2\sqrt 2}\log t+\min ((s')^\delta ,(t-\log t-s')^\delta)$.
Therefore for $r\geq \max(1,r_0^{1/\delta},L/(\sqrt 2-1))$,
if $B(s)\geq f(s)$ $\forall s\in [r,t-\log t-r]$,
we have by Proposition~\ref{prop:globalbound}, \eqref{eq:phiulow1} and \eqref{eq:phiulow2} that
\begin{align} \label{eq:phiupath}
\int_0^{t-\log t} \phi\ast u (t-s,B(s))ds
&\leq 2rM+2\int_r^{(t-\log t)/2}Ms^{-\alpha \delta}ds +\int_r^{t-\log t}L e^{2L}(s+\log t)^{-1}ds \notag \\
&<2rM +\tfrac{2M}{\alpha \delta -1}r^{1-\alpha \delta}+L e^{2L}\log t 
\end{align} 
since $\alpha \delta>1$.
Take $r\geq \max(1, r_0^{1/\delta},L/(\sqrt 2-1))$ and let $x=\sqrt 2 t +\frac{1}{2\sqrt 2}\log t+1$.
Then by the Feynman-Kac formula \eqref{feynmankac},
using \eqref{eq:someulog} and \eqref{eq:phiupath}, for $t$ sufficiently large,
\begin{align} \label{eq:dagger}
u(t,x)&\geq t^{-M-1}e^{t-\log t} \exp\left(-\left(2rM +\tfrac{2M}{\alpha \delta -1}r^{1-\alpha \delta}+L e^{2L}\log t\right)\right) \notag\\
&\quad \mathbb P_x \bigg(B(t-\log t)\in [\tfrac{5}{4}\sqrt 2 \log t+1,\tfrac{5}{4}\sqrt 2 \log t+2], B(s)\geq f(s)\,\,\forall s\in [r,t-\log t-r]\bigg).
\end{align}

We now aim to prove a lower bound for the probability above.
For $y\in [\tfrac{5}{4}\sqrt 2 \log t+1,\tfrac{5}{4}\sqrt 2 \log t+2]$ and
$x=\sqrt 2 t +\frac{1}{2\sqrt 2}\log t+1$,
since $\frac{5}{4}\sqrt 2 = \sqrt 2 +\frac{1}{2\sqrt 2}$,
\begin{align*}
&\psub{x}{B(s)\geq f(s)\,\,\forall s\in [r,t-\log t-r] \bigg| 
B(t-\log t)=y}\\
&=\p{\xi^{t-\log t}(s)+\frac{s}{t-\log t}y+\frac{t-\log t-s}{t-\log t}x\geq f(s)\,\,\forall s\in [r,t-\log t-r] }\\
&=\mathbb P\bigg( \xi^{t-\log t}(s)+\sqrt 2 (t-s)+\tfrac{1}{2\sqrt 2}\log t+
\frac{s}{t-\log t}y'+\frac{t-\log t-s}{t-\log t}\\
&\hspace{9cm}\geq f(s)\quad \forall s\in [r,t-\log t-r] \bigg)
\end{align*}
where $y'=y-(\sqrt2+(2\sqrt 2)^{-1})\log t\in [1,2]$.
Therefore by the definition of $f$ in \eqref{eq:fdef},
\begin{align} \label{eq:light(1)}
&\psub{x}{B(s)\geq f(s)\,\,\forall s\in [r,t-\log t-r] \bigg| 
B(t-\log t)=y} \notag\\
&\geq \p{\xi^{t-\log t}(s)\geq \min(s^\delta,(t-\log t-s)^\delta)\,\,\forall s\in [r,t-\log t-r] }.
\end{align}
By Lemma~\ref{lem:bram2}, there exists $r_{1/2}<\infty$ such that for $r>r_{1/2}$ and $t-\log t>3r$,
\begin{align} \label{eq:light(2)}
&\p{\xi^{t-\log t}(s)\geq \min(s^\delta,(t-\log t-s)^\delta)\,\forall s\in [r,t-\log t-r] } \notag\\
&\geq \tfrac{1}{2}\p{\xi^{t-\log t}(s)\geq -\min(s^\delta,(t-\log t-s)^\delta)\,\forall s\in [r,t-\log t-r] } \notag\\
&\geq \tfrac{1}{2}\p{\xi^{t-\log t}(s)\geq 0\,\forall s\in [r,t-\log t-r] }.
\end{align}
We can now estimate this probability. 
Recall that for $0=s_0<s_1<\ldots <s_n = t$, conditional on $(\xi^t(s_i))_{i=1}^{n-1}$, 
$(\xi^t(s_{i-1}+u),0\leq u \leq s_{i}-s_{i-1})_{i=1}^n$ are independent and
for each $i\in \{1,\ldots, n\}$,
 $(\xi^t(s_{i-1}+u),0\leq u \leq s_{i}-s_{i-1})$ is a Brownian bridge from $\xi^t(s_{i-1})$ to $\xi^t(s_i)$ in time $s_i - s_{i-1}$
(this is the domain Markov property of the Brownian bridge).
We have 
\begin{align*}
&\p{\xi^{t-\log t}(s)\geq 0\,\forall s\in [r,t-\log t-r] }\\
&\geq \p{\xi^{t-\log t}(r)\geq 1,\xi^{t-\log t}(t-\log t-r)\geq 1,\xi^{t-\log t}(s)\geq 0\,\forall s\in [r,t-\log t-r] }\\
&\geq \p{\xi^{t-\log t}(r)\geq 1,\xi^{t-\log t}(t-\log t-r)\geq 1 }\left(1-e^{-2/(t-\log t-2r)}\right)\\
&\geq \p{\xi^{t-\log t}(r)\geq 1}\p{\xi^{t-\log t-r}(r)\geq 1} \left(1-e^{-2/t}\right)\\
&\geq \p{Z\geq \left( \frac{t-\log t-r}{r(t-\log t -2r)}\right) ^{1/2}}^2 (1-e^{-2/t}),
\end{align*}
where the second inequality follows by the domain Markov property of the Brownian bridge and Lemma~\ref{lem:bram1}, the third inequality follows by the domain Markov property, and in the last line $Z\sim N(0,1)$ since $\xi^t(s)\sim N(0,s(t-s)/t)$.
For $r$ sufficiently large, for $t-\log t>3r$,
$\left( \frac{t-\log t-r}{r(t-\log t -2r)}\right) ^{1/2}\leq \Phi^{-1}(3/4),$
so 
\begin{align} \label{eq:light(3)}
\p{\xi^{t-\log t}(s)\geq 0\,\forall s\in [r,t-\log t-r] }
&\geq \tfrac{1}{16} (1-e^{-2/t}) \notag \\
&\geq \tfrac{1}{16}t^{-1},
\end{align}
for $t$ sufficiently large, since $e^{-a}\leq 1-a/2$ for $0\leq a\leq \log 2$.
It follows by combining~\eqref{eq:light(1)},~\eqref{eq:light(2)} and~\eqref{eq:light(3)} that 
\begin{align} \label{eq:estabovef}
&\psub{x}{B(t-\log t)\in [\tfrac{5}{4}\sqrt 2 \log t+1,\tfrac{5}{4}\sqrt 2 \log t+2], B(s)\geq f(s)\,\,\forall s\in [r,t-\log t-r]} \notag\\
&\geq \tfrac{1}{32}t^{-1}\psub{x}{B(t-\log t)\in [\tfrac{5}{4}\sqrt 2 \log t+1,\tfrac{5}{4}\sqrt 2 \log t+2]}.
\end{align}
By \eqref{eq:Btbound}, recalling that $x=\sqrt 2 t +\frac{1}{2\sqrt 2}\log t+1$, we have
\begin{align*}
&\psub{x}{B(t-\log t)\in [\tfrac{5}{4}\sqrt 2 \log t+1,\tfrac{5}{4}\sqrt 2 \log t+2]}\\
&\geq \frac{1}{\sqrt{2\pi t}}\exp\left(-\frac{(\sqrt 2 t +\frac{1}{2\sqrt 2}\log t-\frac{5}{4}\sqrt 2 \log t)^2}{2(t-\log t)}\right)\\
&= \tfrac{1}{\sqrt{2\pi}}e^{-t}t^{1/2}.
\end{align*}
Substituting into \eqref{eq:estabovef} and \eqref{eq:dagger},
if we fix $r$ sufficiently large, then for $t$ sufficiently large and with $t-\log t>3r$, for $x=\sqrt 2 t+\frac{1}{2\sqrt 2 }\log t+1$,
\begin{align*}
u(t,x)&\geq t^{-M-1}e^{t-\log t} \exp\left(-\left(2rM +\tfrac{2M}{\alpha \delta -1}r^{1-\alpha \delta}+L e^{2L}\log t\right)\right) 
\tfrac{1}{32\sqrt{2\pi}}e^{-t}t^{-1/2}\\
&\geq t^{-M-L e^{2L}-3}
\end{align*}
for $t$ sufficiently large.

Hence there exists $T\in (1,\infty)$ such that for $t\geq T$, for $x=\sqrt 2 t+\frac{1}{2\sqrt 2}\log t+1$, $u(t,x)\geq t^{-M-L e^{2L}-3}$, and $T^{-M-L e^{2L}-3}< z_0(1/2)$, where $z_0(1/2)$ is defined in Lemma~\ref{lem:travelu}.
For $t\geq T$, by Lemma~\ref{lem:growu}, letting $C=C(1/2)$ and $R=R(1/2)$, there exist $s_0\in [0,C(M+L e^{2L}+3)\log t]$ and $y_0\in [-R,R]$ such that
$u(t+s_0,x+y_0)\geq 1/2$.
Then by Lemma~\ref{lem:travelu}, for any $s\geq \max(t^*(1),R)$ and $y\in [-R,R]$,
$u(t+s_0+s,x+y_0+y)\geq m^*(1)$.
In particular, for any $t\geq T$ and $s'\geq C(M+L e^{2L}+3)\log t+\max(t^*(1),R)$,
we have
\begin{equation} \label{eq:uhighatx}
u(t+s',\sqrt 2 t+\tfrac{1}{2\sqrt 2 }\log t+1)\geq m^*(1).
\end{equation}

Take $K>\sqrt 2 C(M+L e^{2L}+3)$.
Now for $y\geq \sqrt 2 T+\frac{1}{2\sqrt 2}\log T+1$,
there exists $s=s(y)\geq T$ such that $y=\sqrt 2 s+\frac{1}{2\sqrt 2}\log s+1$.
If $y\leq \sqrt 2 t-K\log t$ for some $t>0$, then $s<t-(K/\sqrt 2)\log t$.
Since $s\geq T$, by \eqref{eq:uhighatx},
for $s'\geq C(M+L e^{2L}+3)\log s+\max(t^*(1),R)$ we have
$u(s+s',y)\geq m^*(1)$.
Hence for $t'\geq t+(C(M+L^2 e^{2L}+3)-(K/\sqrt 2))\log t+\max(t^*(1),R)$ we have
$u(t',y)\geq m^*(1)$.
Since $K>\sqrt 2 C(M+L^2 e^{2L}+3)$, it follows that if $t$ is sufficiently large, for $t'\geq t$,
$u(t',y)\geq m^*(1)$.

We now have that for $t$ sufficiently large, for $y\in [\sqrt 2 T+\frac{1}{2\sqrt 2}\log T+1,\sqrt 2 t -K\log t]$, $u(t,y)\geq m^*(1)$.
Finally, as in the proof of Theorem~\ref{thm:speed}, letting $C=C(1/2)$, $R=R(1/2)$ and $z_0=z_0(1/2)$, by Lemma~\ref{lem:growu}, there exist $s_0\in[0,C\log (1/\min(z_0,u(1,0)))]$ and $y_0\in [-R,R]$ such that $u(1+s_0,y_0)\geq 1/2$.
Then by Lemma~\ref{lem:travelu}, for $t\geq t^*(1)$, for $x\in [y_0-t,y_0+t]$ we have 
$u(1+s_0+t,x)\geq m^*(1)$.
Hence for $t$ sufficiently large,
for $x\in [0,\frac{1}{2}t]$
we have $u(t,x)\geq m^*(1)$.
This completes the proof.
\end{proof}

We can now use Proposition~\ref{prop:logdelay} prove the upper bound of Theorem~\ref{thm:lighttail}, i.e.~that $$\lim_{t\to\infty}\sup_{x\geq \sqrt 2 t -\frac{3}{2\sqrt 2}\log t+10\log \log t}u(t,x)=0.$$
The proof uses ideas from the proof of Proposition~7.3 in \cite{bramson1983}.

\begin{prop} \label{prop:upperboundDlog}
For $\epsilon>0$, there exists $T<\infty$ such that for $t\geq T$, $u(t,x)<\epsilon$ $\forall x\geq \sqrt 2 t -\frac{3}{2\sqrt 2}\log t+10\log \log t$.
\end{prop}

\begin{proof}
Recall that $\phi \geq \eta \text{ a.e.~on }(-\sigma,\sigma).$
Fix $t$ large
and $x\in [\sqrt 2 t -\frac{3}{2\sqrt 2}\log t,\sqrt 2 t]$.
Then for $j\in [0,t-1]$, define the event
\begin{equation} \label{eq:Ejdef}
E_j=\left\{\exists s\in [j,j+1]:
B(s)<\tfrac{t-s}{t}x-\min(s^{1/4},(t-s)^{1/4})\right\}\cap \left\{\inf_{s\in [0,t-(\log t)^5]}B(s)\geq \sigma
 \right\}.
\end{equation}
Let
\begin{equation} \label{eq:Djdef}
D_j=\Esub{x}{\I{E_j}\exp\left(\int_0^{t} \left(1-\phi \ast u (t-s,B(s))\right)ds \right) u_0(B(t))}.
\end{equation}
By Proposition~\ref{prop:logdelay}, for $s$ sufficiently large,
for $y\in [\sigma,\sqrt 2 s-K\log s-\sigma]$,
\begin{equation*} \label{eq:(star)est}
\phi\ast u(s,y)\geq \eta\int_{-\sigma}^{\sigma} u(s,y-y')dy' \geq 2\sigma \eta m^*(1).
\end{equation*}
Therefore for $j,k\geq 0$ with $j+k\leq t$, if $t-(j+k)$ is sufficiently large then if
$B(s)\in [\sigma,\sqrt 2 (t-s)-K\log (t-s)-\sigma]$ $\forall s\in [j,j+k]$, we have 
\begin{equation*} 
\int_j^{j+k}\phi\ast u(t-s,B(s))ds \geq 2\sigma \eta m^*(1)k.
\end{equation*}
Hence for $t$ sufficiently large, for $k\in[0, t-(\log t)^5]$ and $j\in [0,t-(\log t)^5-k]$,
\begin{align*}
D_j
&\leq \Esub{x}{\I{E_j}\I{\exists s \in [j,j+k]\text{ s.t.~}B(s)\geq \sqrt 2 (t-s)-K\log (t-s)-\sigma} e^t u_0(B(t))}\\
&\quad +
\Esub{x}{\I{E_j}\I{B(s)\leq \sqrt 2 (t-s)-K\log (t-s)-\sigma\, \forall s\in [j,j+k]} e^{\int_0^{t} (1-\phi \ast u (t-s,B(s)))ds} u_0(B(t))}\\
&\leq Le^t\psub{x}{E_j \cap \{\exists s \in [j,j+k]\text{ s.t.~}B(s)\geq \sqrt 2 (t-s)-K\log (t-s)-\sigma\}\cap \{B(t)\leq L\}}\\
&\quad +Le^te^{-2\sigma \eta m^*(1) k}\psub{x}{B(t)\leq L},
\end{align*}
since $\|u_0\|_\infty \leq L$, $u_0(y)=0$ $\forall y\geq L$,  $t-(j+k)\geq (\log t)^5$ and $\inf_{s\in[j,j+k]}B(s)\geq \sigma$ on $E_j$.

Let $k=K' \log t$ for some constant $K'$ such that $2\sigma \eta m^*(1)K'>7/2$.
Then for $t$ sufficiently large, for $j\in [0,(t-1)/2]$,
by the definition of $E_j$ in~\eqref{eq:Ejdef} and
since $x\leq \sqrt 2 t$,
\begin{align*}
D_j
&\leq Le^t\mathbb P_{x}\bigg(B(t)\leq L, \inf_{s\in [j,j+1]}B(s)<\sqrt 2(t-j)-j^{1/4},\\
&\hspace{3cm}\sup_{s\in [j,j+K'\log t]}B(s)\geq \sqrt 2 (t-j-K'\log t)-K\log t -\sigma\bigg)\\
&\quad +Le^t t^{-2\sigma \eta m^*(1)K'}\psub{x}{B(t)\leq L}.
\end{align*}
Also for $t$ sufficiently large, for $j\in [(t-1)/2, t-2(\log t)^5]$,
\begin{align*}
D_j
&\leq Le^t\mathbb P_{x}\bigg(B(t)\leq L, \inf_{s\in [j,j+1]}B(s)<\sqrt 2(t-j)-(t-(j+1))^{1/4},\\
&\hspace{3cm}\sup_{s\in [j,j+K'\log t]}B(s)\geq \sqrt 2 (t-j-K'\log t)-K\log t -\sigma\bigg)\\
&\quad +Le^tt^{-2\sigma \eta m^*(1)K'}\psub{x}{B(t)\leq L}.
\end{align*}
For $j\in [16(\log t)^5,t-16(\log t)^5-1]$, if there exist $ s_1\in [j,j+1]$ such that $B(s_1)<\sqrt 2 (t-j)-2(\log t)^{5/4}$ and $s_2\in [j,j+K'\log t]$ such that $B(s_2)\geq \sqrt 2 (t-j)-(\sqrt 2 K'+K)\log t -\sigma$ then 
$|B(s_1)-B(s_2)|\geq 2(\log t)^{5/4}-(\sqrt 2 K'+K)\log t -\sigma$.
Therefore, for $t$ sufficiently large that $2(\log t)^{5/4}-(\sqrt 2 K'+K)\log t -\sigma \geq (\log t)^{5/4}$, for $j\in [16(\log t)^5,t-16(\log t)^5-1]$,
\begin{align} \label{eq:Djbound}
D_j
&\leq Le^t\mathbb P_{x}\bigg(B(t)\leq L, \sup_{s_1,s_2\in [j,j+K'\log t]}|B(s_1)-B(s_2)|\geq (\log t)^{5/4}\bigg) \notag \\
&\quad +Le^tt^{-2\sigma \eta m^*(1)K'}\psub{x}{B(t)\leq L}.
\end{align}

We now aim to estimate each of these probabilities.
For $y\in [-2t,L]$,
\begin{align*}
&\psub{x}{\sup_{s_1,s_2\in [j,j+K'\log t]}|B(s_1)-B(s_2)|\geq (\log t)^{5/4} \bigg| B(t)=y}\\
&\leq \psub{x}{\sup_{s\in [j,j+K'\log t]}|B(s)-B(j)|\geq \tfrac{1}{2}(\log t)^{5/4} \bigg| B(t)=y}\\
&= \p{\sup_{s\in [j,j+K'\log t]}\left|\tfrac{s-j}{t}y+\tfrac{j-s}{t}x+\xi^t(s)-\xi^t(j)\right|\geq \tfrac{1}{2}(\log t)^{5/4} }.
\end{align*}
Since $0\leq x\leq \sqrt 2 t$ and $|y|\leq 2t$,
for $s\in [j,j+K'\log t]$,
$$\left|\tfrac{s-j}{t}y+\tfrac{j-s}{t}x\right|\leq \tfrac{K'\log t}{t}(2t+\sqrt 2 t)\leq \tfrac{1}{4}(\log t)^{5/4}$$
for $t$ sufficiently large.
Therefore
\begin{align*}
&\psub{x}{\sup_{s_1,s_2\in [j,j+K'\log t]}|B(s_1)-B(s_2)|\geq (\log t)^{5/4} \bigg| B(t)=y}\\
&\leq \p{\sup_{s\in [j,j+K'\log t]}|\xi^t(s)-\xi^t(j)|\geq \tfrac{1}{4}(\log t)^{5/4} }\\
&= \psub{0}{\sup_{s\in [j,j+K'\log t]}\left|B(s)-B(j)+\tfrac{j-s}{t}B(t)\right|\geq \tfrac{1}{4}(\log t)^{5/4} }\\
&\leq \psub{0}{\sup_{s\in [j,j+K'\log t]}|B(s)-B(j)|\geq \tfrac{1}{8}(\log t)^{5/4} }+ \psub{0}{\tfrac{K' \log t}{t}|B(t)|\geq \tfrac{1}{8}(\log t)^{5/4} }\\
&\leq 2\psub{0}{\sup_{s\in [0,K'\log t]}B(s)\geq \tfrac{1}{8}(\log t)^{5/4} } + 2\psub{0}{B(t)\geq \tfrac{1}{8K'}t(\log t)^{1/4} },
\end{align*}
where the third line follows since $(B(s)-\tfrac{s}{t}B(t), s\in [0,t])\eqdist (\xi^t(s),s\in [0,t])$ and the last line follows by the symmetry of Brownian motion and the Markov property.
By the reflection principle, it follows that for $y\in [-2t,L]$,
\begin{align} \label{eq:(A)3}
&\psub{x}{\sup_{s_1,s_2\in [j,j+K'\log t]}|B(s_1)-B(s_2)|\geq (\log t)^{5/4} \bigg| B(t)=y} \notag\\
&\leq 4\psub{0}{(K'\log t)^{1/2} B(1) \geq \tfrac{1}{8}(\log t)^{5/4} } + 2\psub{0}{t^{1/2} B(1)\geq \tfrac{1}{8K'}t(\log t)^{1/4} }\notag \\
&\leq 4\exp(-(\log t)^{3/2}/(128 K'))+2\exp (-t(\log t)^{1/2}/(128 (K')^2)),
\end{align}
by~\eqref{eq:gaussiantail}.
Moreover, by~\eqref{eq:gaussiantail},
since $x\geq \sqrt 2 t-\frac{3}{2\sqrt 2}\log t$,
\begin{align} \label{eq:(B)3}
\psub{x}{B(t)\in [-2t,L]}\leq
\psub{x}{B(t)\leq L}
&\leq e^{-(\sqrt 2 t-\frac{3}{2\sqrt 2}\log t-L)^2/(2t)} \notag\\
&\leq e^{-t+\frac{3}{2} \log t+\sqrt 2 L}.
\end{align} 
Also since $x\geq 0$, by~\eqref{eq:gaussiantail},
\begin{align} \label{eq:(C)3}
\psub{x}{B(t)\leq -2t}
\leq e^{-\tfrac{1}{2t}4t^2}
=e^{-2t}.
\end{align} 
Therefore, combining~\eqref{eq:(A)3}, \eqref{eq:(B)3} and \eqref{eq:(C)3}, 
\begin{align*}
&\mathbb P_{x}\bigg(B(t)\leq L, \sup_{s_1,s_2\in [j,j+K'\log t]}|B(s_1)-B(s_2)|\geq (\log t)^{5/4}\bigg)\\
&\hspace{0.5cm}\leq e^{-t+\frac{3}{2} \log t+\sqrt 2 L}(4e^{-(\log t)^{3/2}/(128 K')}+2e^{-t/(128 (K')^2)})+e^{-2t}.
\end{align*}
Substituting into \eqref{eq:Djbound} and using~\eqref{eq:(B)3} again,
it follows that for $t$ sufficiently large,
for $j\in [16(\log t)^5,t-16(\log t)^5-1]$,
\begin{align*}
D_j
&\leq Le^te^{-t+\frac{3}{2} \log t+\sqrt 2 L}(4e^{-(\log t)^{3/2}/(128 K')}+2e^{-t/(128 (K')^2)}+t^{-2\sigma \eta m^*(1)K'})+Le^{-t}\\
&\leq t^{-2}
\end{align*}
for $t$ sufficiently large, since $2\sigma \eta m^*(1)K'>7/2$ by our choice of $K'$.
Therefore by the definition of $D_j$ in \eqref{eq:Djdef} and the Feynman-Kac formula \eqref{feynmankac0}, for $t$ sufficiently large, 
\begin{align*}
u(t,x)&\leq \sum_{j=\lceil 16 (\log t)^5 \rceil}^{\lfloor t-16(\log t)^5-1 \rfloor}D_j \\
&\hspace{1cm}+\Esub{x}{\prod_{j=\lceil 16 (\log t)^5 \rceil}^{\lfloor t-16(\log t)^5-1 \rfloor}\I{E_j^c}\exp\left(\int_0^{t} \left(1-\phi \ast u (t-s,B(s))\right)ds \right) u_0(B(t))}\\
&\leq t\cdot t^{-2}+Le^t\Esub{x}{\prod_{j=\lceil 16 (\log t)^5 \rceil}^{\lfloor t-16(\log t)^5-1 \rfloor}\I{E_j^c}\I{B(t)\leq L}},
\end{align*}
since $\phi\ast u \geq 0$, $\|u_0\|_\infty \leq L$ and $u_0(y)=0$ $\forall y \geq L$.
By the definition of $E_j$ in \eqref{eq:Ejdef}, it follows that for $t$ sufficiently large, for $x\in[\sqrt 2 t -\frac{3}{2\sqrt 2 }\log t,\sqrt 2 t]$,
\begin{align} \label{eq:ubound}
u(t,x)&\leq t^{-1}+Le^t
\psub{x}{\inf_{s\in [0,t-(\log t)^5]}B(s)< \sigma} \notag\\
&\quad +Le^t\mathbb P_{x}\bigg(B(t)\leq L, B(s)\geq \tfrac{t-s}{t}x-\min(s^{1/4},(t-s)^{1/4})\, \notag\\
&\hspace{6cm}\forall s\in [\lceil 16 (\log t)^5 \rceil,\lfloor t-16(\log t)^5\rfloor]\bigg).
\end{align}
Since $x\geq \sqrt 2 t -\frac{3}{2\sqrt 2}\log t$, by the reflection principle and then using \eqref{eq:gaussiantail2},
\begin{align} \label{eq:infBbound}
\psub{x}{\inf_{s\in [0,t-(\log t)^5]}B(s)< \sigma}
&\leq 2\psub{0}{B(t-(\log t)^5)>\sqrt 2 t -\tfrac{3}{2\sqrt 2}\log t-\sigma}\notag\\
&\hspace{-1cm}\leq \frac{2}{\sqrt{2\pi}}\left(\frac{t-(\log t)^5}{(\sqrt 2 t -\frac{3}{2\sqrt 2}\log t-\sigma)^2}\right)^{1/2}\exp\left(-\frac{(\sqrt 2 t -\frac{3}{2\sqrt 2}\log t-\sigma)^2}{2(t-(\log t)^5)} \right)\notag\\
&\hspace{-1cm}\leq \frac{2}{\sqrt{2\pi}}t^{-1/2}\left(\frac{t-(\log t)^5}{2t -3\log t-2\sqrt 2 \sigma}\right)^{1/2}\exp\left(-t\frac{2t-3\log t-2\sqrt 2 \sigma}{2t-2(\log t)^5} \right)\notag\\
&\hspace{-1cm}\leq t^{-1/2}e^{-t}
\end{align}
for $t$ sufficiently large.

We shall now estimate the second probability in~\eqref{eq:ubound}.
For $y\leq L$,
and $x\in[\sqrt 2 t -\frac{3}{2\sqrt 2 }\log t,\sqrt 2 t]$,
\begin{align*}
&\psub{x}{B(s)\geq \tfrac{t-s}{t}x-\min(s^{1/4},(t-s)^{1/4})\,
\forall s\in [\lceil 16 (\log t)^5 \rceil,\lfloor t-16(\log t)^5\rfloor]\bigg| B(t)=y}\\
&=\p{\xi^t(s)\geq -\tfrac{s}{t}y-\min(s^{1/4},(t-s)^{1/4})\,
\forall s\in [\lceil 16 (\log t)^5 \rceil,\lfloor t-16(\log t)^5\rfloor]}\\
&\leq\p{\xi^t(s)\geq -L-\min(s^{1/4},(t-s)^{1/4})\,
\forall s\in [16 (\log t)^5 +1,t-16(\log t)^5-1]}\\
&\leq 2\p{\xi^t(s)\geq -L+\min(s^{1/4},(t-s)^{1/4})\,
\forall s\in [16 (\log t)^5 +1,t-16(\log t)^5-1]}
\end{align*}
for $t$ sufficiently large, by Lemma~\ref{lem:bram2}.
It follows that 
\begin{align} \label{eq:star}
&\psub{x}{B(s)\geq \tfrac{t-s}{t}x-\min(s^{1/4},(t-s)^{1/4})\,
\forall s\in [\lceil 16 (\log t)^5 \rceil,\lfloor t-16(\log t)^5\rfloor]\bigg| B(t)=y} \notag\\
&\leq 2\p{\xi^t(s)\geq -L\,
\forall s\in [16 (\log t)^5 +1,t-16(\log t)^5-1]}\notag \\
&\leq 2\mathbb P\big(\xi^t(s)\geq -L\,
\forall s\in [16 (\log t)^5 +1,t-16(\log t)^5-1],\notag \\
&\hspace{3cm}\xi^t(16 (\log t)^5 +1)\leq (\log t)^5,
\xi^t(t-16 (\log t)^5 -1)\leq (\log t)^5\big)\notag \\
&\hspace{0.5cm}+4\p{\xi^t(16 (\log t)^5 +1)\geq (\log t)^5},
\end{align}
since $\xi^t(16 (\log t)^5 +1)\stackrel{d}{=}\xi^t(t-16 (\log t)^5 -1)$.
For $y_1$, $y_2\in [-L,(\log t)^5]$, by Lemma~\ref{lem:bram1} and the domain Markov property for the Brownian bridge,
\begin{align*}
&\mathbb P\bigg(\xi^t(s)\geq -L\,
\forall s\in [16 (\log t)^5 +1,t-16(\log t)^5-1] \bigg|\\
&\hspace{3cm}\xi^t(16 (\log t)^5 +1)=y_1,
\xi^t(t-16 (\log t)^5 -1)=y_2\bigg)\\
&=1-\exp\left(-\frac{2}{t-32(\log t)^5-2}(y_1+L)(y_2+L) \right)\\
&\leq \frac{2}{t-32(\log t)^5-2}((\log t)^5+L)^2,
\end{align*}
since $1-e^{-a}\leq a$ for $a\geq 0$.

Since $\xi^t(s)\sim N(0,\frac{s(t-s)}{t})$, letting $Z\sim N(0,1)$,
\begin{align*}
\p{\xi^t(16 (\log t)^5 +1)\geq (\log t)^5}
&\leq \p{(16 (\log t)^5 +1)^{1/2}Z\geq (\log t)^5}\\
&\leq \exp(-\tfrac{1}{2}(\log t)^{10}(16 (\log t)^5 +1)^{-1})\\
&\leq \exp(-\tfrac{1}{40}(\log t)^{5})
\end{align*}
for $t$ sufficiently large, where the second inequality holds by \eqref{eq:gaussiantail}.
Substituting into \eqref{eq:star}, for any $y\leq L$ we have 
\begin{align*}
&\psub{x}{B(s)\geq \tfrac{t-s}{t}x-\min(s^{1/4},(t-s)^{1/4})\,
\forall s\in [\lceil 16 (\log t)^5 \rceil,\lfloor t-16(\log t)^5\rfloor]\bigg| B(t)=y} \\
&\leq \frac{4}{t-32(\log t)^5-2}((\log t)^5+L)^2+4\exp(-\tfrac{1}{40}(\log t)^{5})\\
&\leq \frac{5(\log t)^{10}}{t}
\end{align*}
for $t$ sufficiently large.
Substituting this result and \eqref{eq:infBbound} into \eqref{eq:ubound}, we have that for $t$ sufficiently large, for $x\in [\sqrt 2 t -\frac{3}{2\sqrt 2}\log t,\sqrt 2 t]$,
\begin{align*}
u(t,x)&\leq t^{-1}+Le^t
t^{-1/2}e^{-t} 
 +Le^t\frac{5(\log t)^{10}}{t}\psub{x}{B(t)\leq L}.
\end{align*}
We have $x=\sqrt 2 t -\frac{3}{2\sqrt 2}\log t+x_0$ for some $x_0\in [0,\frac{3}{2\sqrt 2}\log t]$; by \eqref{eq:gaussiantail2},
$$
\psub{x}{B(t)\leq L}
\leq \tfrac{1}{\sqrt{2 \pi}} z^{-1}e^{-z^2/2},
$$
where $z=(\sqrt 2 t -\frac{3}{2\sqrt 2}\log t+x_0-L)t^{-1/2} \geq t^{1/2}$ for $t$ sufficiently large.
Hence 
\begin{align*}
\psub{x}{B(t)\leq L}&\leq \frac{1}{\sqrt{2\pi t}}\exp\left(-(\sqrt 2 t -\tfrac{3}{2\sqrt 2}\log t+x_0-L)^2/(2t)\right)\\
&\leq \frac{1}{\sqrt{2\pi t}}\exp\left(-(2t^2-3t\log t +2\sqrt 2 t x_0-2\sqrt 2 tL-\tfrac{3}{\sqrt 2}x_0\log t-2x_0L)/(2t)\right)\\
&\leq \frac{1}{\sqrt{2\pi t}}\exp\left(-t+\tfrac{3}{2}\log t -\sqrt 2 x_0+\sqrt 2 L+1\right)
\end{align*}
for $t$ sufficiently large. Therefore, for $t$ sufficiently large,
for $x=\sqrt 2 t -\frac{3}{2\sqrt 2}\log t+x_0$ with $x_0\in [0,\frac{3}{2\sqrt 2}\log t]$,
\begin{align*}
u(t,x)&\leq t^{-1}+Lt^{-1/2} 
 +L\frac{5(\log t)^{10}}{t}\frac{1}{\sqrt{2\pi t}}\exp\left(\tfrac{3}{2}\log t -\sqrt 2 x_0+\sqrt 2 L+1\right)\\
&=t^{-1}+Lt^{-1/2} 
 +\frac{5 L e^{\sqrt 2 L+1}}{\sqrt{2\pi }}(\log t)^{10}e^{-\sqrt 2 x_0}.
\end{align*}
It follows that if $x_0\in [10\log \log t,\frac{3}{2\sqrt 2}\log t]$,
i.e.~if $x\in [\sqrt 2 t -\frac{3}{2\sqrt 2}\log t+10\log\log t,\sqrt 2 t]$,
 then 
\begin{align*}
u(t,x)&\leq t^{-1}+Lt^{-1/2} 
 +\frac{5L e^{\sqrt 2 L+1}}{\sqrt{2\pi }}(\log t)^{10(1-\sqrt 2)}.
\end{align*}
As in \eqref{uabovesqrt2} in the proof of Theorem~\ref{thm:speed}, for $t\geq L/(\sqrt 2-1)$, for $x\geq \sqrt 2 t$,
$$
u(t,x)\leq \frac{L}{\sqrt{2\pi t}}e^{\sqrt 2 L}.
$$
Hence for $\epsilon>0$ fixed, for $t$ sufficiently large, for any $x\geq \sqrt 2 t -\frac{3}{2\sqrt 2}\log t+10\log \log t$,
we have 
$u(t,x)<\epsilon$, as required.
\end{proof}
As a consequence of Proposition~\ref{prop:upperboundDlog} and Lemma~\ref{lem:growu}, we can adapt an idea from the proof of Proposition~7.2 in \cite{bramson1983} to show that for fixed $t$, $u(t,x)$ decreases exponentially on $x\geq \sqrt 2 t-\frac{3}{2\sqrt 2}\log t+10\log \log t$.

Let $z_0=z_0(1/2)$, $R=R(1/2)$ and $C=C(1/2)$ as defined in Lemma~\ref{lem:growu}.
\begin{cor} \label{cor:expdecay}
There exists $T<\infty$ such that for $t\geq T$ and $s> C\log (1/z_0)$,
if $x\geq \sqrt 2 (t+s)-\frac{3}{2\sqrt 2}\log (t+s)+10\log \log (t+s)+R$
then $u(t,x)\leq \exp(-C^{-1}s)$.
\end{cor}
\begin{proof}
Let $g(t):=\sqrt 2 t-\frac{3}{2\sqrt 2}\log t+10\log \log t$ for $t>1$, and note that $g$ is increasing on $(1,\infty)$.
Take $\epsilon=1/2$ in Proposition~\ref{prop:upperboundDlog}; it follows that there exists $T\in (1,\infty)$ such that for $t\geq T$, we have $u(t,x)<1/2$ $\forall x\geq g(t)$.

We now use a contradiction argument.
Suppose that for some $t\geq T$ and $s> C\log (1/z_0)$, there exists $x\geq g(t+s)+R$
such that $u(t,x)>\exp(-C^{-1}s)$.
Then since $\exp(-C^{-1}s)< z_0$,
by Lemma~\ref{lem:growu} there exist
$t'\in [t,t+s]$ and $y\in [x-R,x+R]$ such that $u(t',y)\geq 1/2$.
Since $x\geq g(t+s)+R$,
we have that
$
y\geq g(t+s) \geq g(t')
$
since $g$ is increasing on $(1,\infty)$.
But then we have $t'\geq T$, $y\geq g(t')$ and $u(t',y)\geq 1/2$ which is a contradiction.
\end{proof}

We can now prove the remaining statement of Theorem~\ref{thm:lighttail}, i.e.~that 
there exist $A<\infty$ and $m^*>0$ such that
$$
\liminf_{t\to\infty}\inf_{x\in[0,\sqrt 2 t-\frac{3}{2\sqrt 2}\log t-A(\log \log t)^3]}u(t,x)\geq m^*.
$$
\begin{prop} \label{prop:loglowerbound}
There exist $A<\infty$ and $T<\infty$ such that for $t\geq T$, 
$u(t,x)\geq m^*(1)$
$\forall x\in [0,\sqrt 2 t-\frac{3}{2\sqrt 2}\log t-A(\log \log t)^3]$.
\end{prop}
\begin{proof}
Take $\delta\in (\max(1/\alpha,1/3),1/2)$ and take $t$ large.
For $s\in (e, t/2]$ sufficiently large that $s^\delta\geq 10\log \log (s+s^\delta)+R$,
if $x\geq \frac{s}{t}(\sqrt 2 t-\frac{3}{2\sqrt 2}\log t)+(\sqrt 2 +1)s^\delta$,
then
\begin{align*}
x&\geq \sqrt 2 (s+s^\delta)-\tfrac{3}{2\sqrt 2}\tfrac{s}{t}\log t+s^\delta\\
&\geq \sqrt 2 (s+s^\delta)-\tfrac{3}{2\sqrt 2}\log s+s^\delta\\
&\geq \sqrt 2 (s+s^\delta)-\tfrac{3}{2\sqrt 2}\log (s+s^\delta)+s^\delta\\
&\geq \sqrt 2 (s+s^\delta)-\tfrac{3}{2\sqrt 2}\log (s+s^\delta)+10\log \log (s+s^\delta)+R,
\end{align*}
where the second line holds since for $e<s<t$, $\frac{\log s}{s}>\frac{\log t}{t}$.
By Corollary~\ref{cor:expdecay}, if $s$ is sufficiently large and $s^\delta> C\log (1/z_0)$, it follows that $u(s,x)\leq \exp(-C^{-1}s^\delta)$.

Similarly, for $s\in [t/2,t-(\log \log t)^{1/\delta}]$,
if $x\geq \frac{s}{t}(\sqrt 2 t-\frac{3}{2\sqrt 2}\log t)+(\sqrt 2 +12)(t-s)^\delta$,
then
\begin{align*}
x&\geq \sqrt 2 (s+(t-s)^\delta)-\tfrac{3}{2\sqrt 2}\tfrac{s}{t}\log t+12(t-s)^\delta\\
&\geq \sqrt 2 (s+(t-s)^\delta)-\tfrac{3}{2\sqrt 2}\log (s+(t-s)^\delta)+12(t-s)^\delta\\
&\geq \sqrt 2 (s+(t-s)^\delta)-\tfrac{3}{2\sqrt 2}\log (s+(t-s)^\delta)+10\log \log (s+(t-s)^\delta)+R,
\end{align*}
where the last line holds for $t$ sufficiently large, since $12(t-s)^\delta \geq 12\log \log t$.
By Corollary~\ref{cor:expdecay}, if $t$ is sufficiently large and $\log \log t> C\log (1/z_0)$ then it follows that $u(s,x)\leq \exp(-C^{-1}(t-s)^\delta)$.

Therefore, for $t$ sufficiently large, for $s\in [(\log \log t)^{1/\delta},t-(\log \log t)^{1/\delta}]$,
if $x\geq \frac{s}{t}(\sqrt 2 t-\frac{3}{2\sqrt 2}\log t)+(\sqrt 2 +12)\min(s^{\delta},(t-s)^\delta)$
then
$u(s,x)\leq \exp(-C^{-1}\min(s^\delta, (t-s)^\delta)).$
It follows that for $s\in [(\log \log t)^{1/\delta}, t/2]$, if $x\geq \frac{s}{t}(\sqrt 2 t-\frac{3}{2\sqrt 2}\log t)+(\sqrt 2 +13)s^{\delta}$
then by Proposition~\ref{prop:globalbound},
\begin{align*}
\phi \ast u(s,x)
&\leq M\int_{s^{\delta}}^\infty \phi(r)dr +\exp(-C^{-1}s^{\delta})\int^{s^{\delta}}_{-\infty}\phi(r)dr\\
&\leq Ms^{-\alpha \delta}+\exp(-C^{-1}s^{\delta})\\
&\leq 2Ms^{-\alpha \delta}
\end{align*}
for $t$ sufficiently large,
where the second inequality follows for $t$ sufficiently large by~\eqref{eq:lighttail} and since $\phi\geq 0$ and $\int_{-\infty}^\infty \phi(x)dx=1$.
Similarly, for $t$ sufficiently large, for $s\in [t/2,t-(\log \log t)^{1/\delta}]$, if $x\geq \frac{s}{t}(\sqrt 2 t-\frac{3}{2\sqrt 2}\log t)+(\sqrt 2 +13)(t-s)^{\delta}$
then
\begin{align*}
\phi \ast u(s,x)
&\leq M(t-s)^{-\alpha \delta}+\exp(-C^{-1}(t-s)^{\delta})\leq 2M(t-s)^{-\alpha \delta}.
\end{align*}
Therefore, if $B(s)\geq \frac{t-s}{t}(\sqrt 2 t-\frac{3}{2\sqrt 2}\log t)+(\sqrt 2 +13)\min(s^{\delta},(t-s)^\delta)$
$\forall s\in [(\log \log t)^{1/\delta},t-(\log \log t)^{1/\delta}]$,
then by the above inequalities and Proposition~\ref{prop:globalbound},
\begin{align*}
\int_0^{t-1} \phi \ast u (t-s,B(s))ds &\leq 2M(\log \log t)^{1/\delta}
+2\int_{(\log \log t)^{1/\delta}}^{t/2}2Ms^{-\alpha \delta}ds\\
&=2M(\log \log t)^{1/\delta}+4M(\alpha \delta-1)^{-1}((\log \log t)^{(1-\alpha \delta)/\delta}-(t/2)^{1-\alpha \delta})\\
&\leq 3M(\log \log t)^{1/\delta}
\end{align*}
for $t$ sufficiently large, since $\alpha \delta >1$.

It follows that for $t$ sufficiently large, for $x=\sqrt 2 t - \frac{3}{2\sqrt 2}\log t$, by the Feynman-Kac formula \eqref{feynmankac} and since $u\geq 0$,
\begin{align*}
u(t,x)&\geq
e^{t-1}e^{-3M(\log \log t)^{1/\delta}}\\
&\qquad \Esub{x}{\I{B(s)\geq \frac{t-s}{t}x+(\sqrt 2 +13)\min(s^{\delta},(t-s)^\delta)
\forall s\in [(\log \log t)^{1/\delta},\, t-(\log \log t)^{1/\delta}]}
u(1,B(t-1))}.
\end{align*}
As in the proof of Proposition~\ref{prop:logdelay}, we can take $\epsilon\in (0,\min(\frac{1}{2}u(1,0),1))$ sufficiently small that Lemma~\ref{lem:abscty} holds for this choice of $\epsilon$, and then
for $|y|\leq \epsilon^3$, we have $u(1,y)\geq\epsilon$.
Hence letting $t'=(\log \log t)^{1/\delta}$,
\begin{align} \label{eq:ulower}
u(t,x)&\geq
e^{t-1}e^{-3M(\log \log t)^{1/\delta}}\epsilon \notag\\
&\qquad  \psub{x}{B(s)\geq \tfrac{t-s}{t}x+(\sqrt 2 +13)\min(s^{\delta},(t-s)^\delta)
\,\,\forall s\in [t',t-t'],|B(t-1)|\leq \epsilon^3}.
\end{align}
We now aim for a lower bound on this probability.
For $y\in [-\epsilon^3,\epsilon^3]$,
\begin{align*}
&\psub{x}{B(s)\geq \tfrac{t-s}{t}x+(\sqrt 2 +13)\min(s^{\delta},(t-s)^\delta)
\,\,\forall s\in [t',t-t']\bigg|B(t-1)=y}\\
&= \p{\xi^{t-1}(s)+\tfrac{t-1-s}{t-1}x+\tfrac{s}{t-1}y\geq \tfrac{t-s}{t}x+(\sqrt 2 +13)\min(s^{\delta},(t-s)^\delta)
\,\,\forall s\in [t',t-t']}\\
&= \p{\xi^{t-1}(s)\geq \tfrac{s}{t(t-1)}x-\tfrac{s}{t-1}y+(\sqrt 2 +13)\min(s^{\delta},(t-s)^\delta)
\,\,\forall s\in [t',t-t']}\\
&\geq \p{\xi^{t-1}(s)\geq \sqrt 2+2+(\sqrt 2 +13)\min(s^{\delta},(t-1-s)^\delta)
\,\,\forall s\in [t'-1,t-t']},
\end{align*}
for $t$ sufficiently large that $\min(s^{\delta},(t-s)^\delta)-\min(s^{\delta},(t-1-s)^\delta)\leq (\sqrt 2 +13)^{-1}$ $\forall s\in [t',t-t']$, and since $x/t<\sqrt 2$ and $|y|<1$.
By Lemma~\ref{lem:bram2}, it follows that for $t$ sufficiently large,
\begin{align*}
&\psub{x}{B(s)\geq \tfrac{t-s}{t}x+(\sqrt 2 +13)\min(s^{\delta},(t-s)^\delta)
\,\,\forall s\in [t',t-t']\bigg|B(t-1)=y}\\
&\geq \tfrac{1}{2}\p{\xi^{t-1}(s)\geq \sqrt 2+2-(\sqrt 2 +13)\min(s^{\delta},(t-1-s)^\delta)
\,\,\forall s\in [t'-1,t-t']}\\
&\geq \tfrac{1}{2}\p{\xi^{t-1}(s)\geq \sqrt 2+2
\,\,\forall s\in [1,t-2]}.
\end{align*}
By the domain Markov property of the Brownian bridge and Lemma~\ref{lem:bram1}, letting $Z\sim N(0,1)$,
\begin{align*}
&\p{\xi^{t-1}(s)\geq \sqrt 2+2\,
\,\,\forall s\in [1,t-2]}\\
&\qquad\geq \p{\xi^{t-1}(1)\geq \sqrt 2 +3, \xi^{t-1}(t-2)\geq \sqrt 2 +3}(1-e^{-2/(t-3)})\\
&\qquad\geq \p{\xi^{t-1}(1)\geq \sqrt 2 +3}\p{\xi^{t-2}(1)\geq \sqrt 2 +3}(1-e^{-2/t})\\
&\qquad \geq \tfrac{1}{t}\p{\left(\tfrac{t-3}{t-2}\right)^{1/2}Z\geq \sqrt 2 +3}^2.
\end{align*}
for $t$ sufficiently large, since $1-e^{-a}\geq a/2$ for $0<a<\log 2$
and $\xi^t(s)\sim N\left(0,\frac{s(t-s)}{t}\right)$.
Therefore for $t$ sufficiently large, for $y\in [-\epsilon^3,\epsilon^3]$,
\begin{align*}
&\psub{x}{B(s)\geq \tfrac{t-s}{t}x+(\sqrt 2 +13)\min(s^{\delta},(t-s)^\delta)
\,\,\forall s\in [t',t-t']\bigg|B(t-1)=y}\\
&\geq \tfrac{1}{2t}\Phi(-5)^2.
\end{align*}
Substituting into \eqref{eq:ulower}, we have that for $t$ sufficiently large, for $x=\sqrt 2 t-\frac{3}{2\sqrt 2}\log t$, 
\begin{align*}
u(t,x)&\geq
e^{t-1}e^{-3M(\log \log t)^{1/\delta}}\epsilon \tfrac{1}{2t}\Phi(-5)^2 \psub{x}{|B(t-1)|\leq \epsilon^3}.
\end{align*}
Since $x=\sqrt 2 t -\frac{3}{2\sqrt 2}\log t$ and $\epsilon<1$, we have by \eqref{eq:Btbound} that
\begin{align*}
\psub{x}{|B(t-1)|\leq \epsilon^3}
&\geq \frac{2\epsilon^3}{\sqrt{2\pi t}}\exp\left(-\tfrac{1}{2(t-1)}(\sqrt 2 t -\tfrac{3}{2\sqrt 2}\log t+1)^2\right)\\
&\geq \frac{2\epsilon^3}{\sqrt{2\pi t}}\exp\left(-t+\tfrac{3}{2}\log t-\sqrt 2-1\right)
\end{align*}
for $t$ sufficiently large. Hence for $t$ sufficiently large,
\begin{align*}
u(t,\sqrt 2 t -\tfrac{3}{2\sqrt 2}\log t)
&\geq 
\tfrac{1}{\sqrt{2\pi }}e^{-\sqrt 2-2}\epsilon^4 \Phi(-5)^2e^{-3M(\log \log t)^{1/\delta}} \\
&\geq e^{-4M(\log \log t)^{1/\delta}}
\end{align*}
for $t$ sufficiently large.

For $t$ sufficiently large, it follows by Lemma~\ref{lem:growu} that there exist $t'\in [t,t+4CM(\log \log t)^{1/\delta}]$ and $y\in [-R,R]$ such that
$u(t',\sqrt 2 t -\tfrac{3}{2\sqrt 2}\log t+y)\geq 1/2$.
Then by Lemma~\ref{lem:travelu}, for $s\geq t^*(1)+R+4CM(\log \log t)^{1/\delta}$ we have $u(t+s,\sqrt 2 t -\tfrac{3}{2\sqrt 2}\log t)\geq m^*(1)$.

Note that $t\mapsto \sqrt 2 t-\frac{3}{2\sqrt 2}\log t$ is increasing on $[1,\infty)$.
Take $t>1$ sufficiently large that 
\begin{equation} \label{eq:t_cond}
CM(\log \log t)^{1/\delta}\geq t^*(1)+R
\end{equation}
and take
$y_0\in [\sqrt 2,\sqrt 2 t-\frac{3}{2\sqrt 2}\log t-5\sqrt 2 CM(\log \log t)^{1/\delta}]$; then 
$y_0=\sqrt 2 s_0 -\frac{3}{2\sqrt 2}\log s_0$ for some $s_0\in [1,t)$.
Then if $s_0$ is sufficiently large, for $s\geq t^*(1)+R+4CM(\log \log s_0)^{1/\delta}$ we have $u(s_0+s,y_0)\geq m^*(1)$.
Since 
$$y_0=\sqrt 2 s_0 -\tfrac{3}{2\sqrt 2}\log s_0\leq \sqrt 2 t-\tfrac{3}{2\sqrt 2}\log t-5\sqrt 2 CM(\log \log t)^{1/\delta},$$
we have 
$$\sqrt 2 (t-s_0)\geq \tfrac{3}{2\sqrt 2}(\log t-\log s_0)+5\sqrt 2 CM(\log \log t)^{1/\delta}\geq 5\sqrt 2 CM(\log \log t)^{1/\delta}.$$
Therefore, by our choice of $t$ in~\eqref{eq:t_cond}, $t-s_0\geq t^*(1)+R+4CM(\log \log t)^{1/\delta}.$
It follows that for $t$ sufficiently large, for $y_0\in [\sqrt 2,\sqrt 2 t-\frac{3}{2\sqrt 2}\log t-5\sqrt 2 CM(\log \log t)^{1/\delta}]$ sufficiently large,
$u(t,y_0)\geq m^*(1)$.

It follows by
Proposition~\ref{prop:logdelay} that for $t$ sufficiently large, for any $y\in [0,\sqrt 2 t-\frac{3}{2\sqrt 2}\log t-5\sqrt 2 CM(\log \log t)^{1/\delta}]$, we have 
$u(t,y)\geq m^*(1)$.
Since $1/\delta\leq 3$, this completes the proof.
\end{proof}

\section{Proofs of Theorems~\ref{thm:heavylower}--\ref{thm:heavyweak}} \label{sec:heavy}
Before proving Theorems~\ref{thm:heavylower}--\ref{thm:heavyweak}, we need an estimate on the probability that a Brownian motion stays consistently ahead of a particular curve. 
We shall use the following result from \cite{roberts2015}.
\begin{prop}[Simplified version of Proposition 4 in \cite{roberts2015}]
\label{prop:roberts}
There exists a function $A:(0,\infty)\to (0,\infty)$ such that the following holds.
Suppose $f:[0,t]\to \R$ and $L:[0,t]\to [1,\infty)$ are twice continuously differentiable, with $f(0)<0$ and $f(0)+L(0)>0$.
Also suppose there exists a constant $Q>0$ such that 
\begin{align*}
&|L'(0)|L(0)+|L'(t)|L(t)+\int_0^t |L''(s)|L(s)ds +\int_0^t|f''(s)|L(s)ds\\
&\hspace{3cm}-|L'(0)|f(0)-|f'(0)|f(0)+\log L(0)+|f'(t)|L(t)
\leq Q,
\end{align*}
and that $\int_0^t L(s)^{-2}ds \geq 1$.
Then for any $0\leq p <q\leq 1$,
\begin{align*}
&\psub{0}{B(s)-f(s)\in (0,L(s))\,\,\forall s\leq t, B(t)-f(t)\in (pL(t),qL(t))}\\
&\hspace{0.5cm}
\geq A(Q)\exp\left(-\frac{1}{2}\int_0^t f'(s)^2ds-\frac{\pi^2}{2}\int_0^t \frac{1}{L(s)^{2}}ds \right)
\sin\left(\frac{-\pi f(0)}{L(0)} \right)
\int_p^q \sin(\pi \nu) d\nu .
\end{align*}
\end{prop}
We can use this to prove the following estimate.
\begin{lem} \label{lem:abovegamma}
Suppose $\gamma\in (\frac{1}{2},1)$ and $\epsilon>0$.
Then there exists a constant $c=c(\gamma,\epsilon)>0$ such that for $t$ sufficiently large, for $x\in [-2 t,2 t]$, $\delta \in (0,1)$,
\begin{align*}
&\psub{0}{B(s)\geq \tfrac{s}{t}x +\min(2s^\gamma, 2(t-s)^\gamma)\,\forall s\in [1,\tfrac{t}{2}-1]\cup [\tfrac{t}{2}+1,t-1], |B(t)-x|\leq \delta}\\
&\hspace{0.5cm}\geq c(1-\cos(\tfrac{1}{2}\pi \delta))\exp\left(-\frac{x^2}{2t}-\frac{\pi^2}{2\gamma+2\epsilon-1}t^{2\gamma+2\epsilon-1} \right).
\end{align*}
\end{lem}

\begin{proof}
By reducing $\epsilon$ if necessary, assume that $\gamma +\epsilon<1$.
For $t\geq 4$, $x\in [-2 t, 2 t]$ and $\delta\in (0,1)$,
\begin{align} \label{eq:star_heavy}
&\psub{0}{B(s)\geq \tfrac{s}{t}x +\min(2s^\gamma, 2(t-s)^\gamma)\,\forall s\in [1,\tfrac{t}{2}-1]\cup [\tfrac{t}{2}+1,t-1], |B(t)-x|\leq \delta} \notag \\
&\geq \psub{0}{B(s)-f(s)\in (0,L(s))\,\,\forall s\leq t, B(t)-f(t)\in (0,\delta)}
\end{align}
for any functions $f:[0,t]\to \R$ and $L:[0,t]\to [1,\infty)$ such that 
$f(t)=x$ and 
$f(s)\geq \frac{s}{t}x+\min(2s^\gamma, 2(t-s)^\gamma)$ $\forall s\in [1,\tfrac{t}{2}-1]\cup [\tfrac{t}{2}+1,t-1]$.
There exists a constant $c_1$ such that for any $t\geq 4$, $x\in [- 2 t,2 t]$, we can define $f:[0,t]\to \R$ in such a way that $f$ is twice continuously differentiable, $f(0)=-1$, $f(t)=x$,
$|f'(s)|\leq c_1$ $\forall s\in [0,t]$, $|f''(s)|\leq c_1$ $\forall s\in [0,t]$,
\begin{align*}
f(s)&=\tfrac{s}{t}x+\min(2s^\gamma, 2(t-s)^\gamma) \, \forall s\in [1,\tfrac{t}{2}-1]\cup [\tfrac{t}{2}+1,t-1]
\end{align*} 
and $|f''(s)|\leq c_1 t^{\gamma-1}$ $\forall s\in [\tfrac{t}{2}-1,\tfrac{t}{2}+1]$.
There exists a constant $c_2$ such that for any $t\geq 4$ we can define $L:[0,t]\to [1,\infty)$ in such a way that $L$ is twice continuously differentiable with $L(0)=2$, $L(t)=2$, $L(s)\leq c_2$ $\forall s\in [0,1]\cup [t-1,t]$, $|L'(s)|\leq c_2$ $\forall s\in [0,t]$, $|L''(s)|\leq c_2$ $\forall s\in [0,t]$,  
$$
L(s)=\min(2s^{1-\gamma-\epsilon},2(t-s)^{1-\gamma-\epsilon})\, \forall s\in [1,\tfrac{t}{2}-1]\cup [\tfrac{t}{2}+1,t-1],
$$
$L(s)\leq 2t^{1-\gamma-\epsilon}$ $\forall s\in [\frac{t}{2}-1,\frac{t}{2}+1]$ and
$|L''(s)|\leq c_2 t^{-\gamma-\epsilon}$ $\forall s\in [\frac{t}{2}-1,\frac{t}{2}+1]$.
We now need to check that the conditions of Proposition~\ref{prop:roberts} hold for $f$ and $L$ for some constant $Q$.
Note that for $s\in[1,\frac{t}{2}-1]$,
we have
$f''(s)=-2\gamma (1-\gamma)s^{\gamma-2}$
and for $s\in[\frac{t}{2}+1,t-1]$,
$f''(s)=-2\gamma (1-\gamma)(t-s)^{\gamma-2}$.
Also for $s\in[1,\frac{t}{2}-1]$,
$L''(s)=-2(\gamma+\epsilon) (1-\gamma-\epsilon)s^{-\gamma-\epsilon-1}$
and for $s\in[\frac{t}{2}+1,t-1]$,
$L''(s)=-2(\gamma+\epsilon) (1-\gamma-\epsilon)(t-s)^{-\gamma-\epsilon-1}$. 
Hence
\begin{align*}
&|L'(0)|L(0)+|L'(t)|L(t)+\int_0^t |L''(s)|L(s)ds +\int_0^t|f''(s)|L(s)ds\\
&\hspace{3cm}-|L'(0)|f(0)-|f'(0)|f(0)+\log L(0)+|f'(t)|L(t)\\
&\leq 4c_2+2\int_1^{t/2-1}4(\gamma+\epsilon) (1-\gamma-\epsilon)s^{-2\gamma-2\epsilon}ds+2c_2^2+4c_2 t^{1-2\gamma -2\epsilon}\\
&\hspace{1cm}+
2\int_1^{t/2-1}4\gamma (1-\gamma)s^{-1-\epsilon}ds+2c_1c_2
+4c_1 t^{-\epsilon}+c_2+c_1+\log 2+2 c_1\\
&< \log 2+7c_1+9c_2+2c_1c_2+2c_2^2
+\frac{8(\gamma+\epsilon) (1-\gamma-\epsilon)}{2\gamma+2\epsilon-1} 
+\frac{8\gamma(1-\gamma)}{\epsilon},
\end{align*}
since $\gamma+\epsilon\in (1/2,1)$ and $\epsilon>0$.
Let
$$
Q=\log 2+7c_1+9c_2+2c_1c_2+2c_2^2
+\frac{8(\gamma+\epsilon) (1-\gamma-\epsilon)}{2\gamma+2\epsilon-1} 
+\frac{8\gamma(1-\gamma)}{\epsilon}.
$$
We also have 
$$
\int_1^{t/2-1}L(s)^{-2}ds=\frac{1}{4}\int_1^{t/2-1}s^{2\gamma+2\epsilon-2}ds=\frac{1}{4}\frac{1}{2\gamma+2\epsilon-1}\left(\left(\tfrac{t}{2}-1\right)^{2\gamma+2\epsilon-1}-1\right),
$$
so since $\gamma+\epsilon>1/2$, we have $\int_0^t L(s)^{-2}ds \geq 1$ for $t$ sufficiently large.
Therefore, for $t$ sufficiently large, Proposition~\ref{prop:roberts} with constant $Q$ applies to $f$ and $L$.

We now need to estimate $\int_0^{t}L(s)^{-2}ds$ and $\int_0^{t}f'(s)^{2}ds$.
Since $L(s)\geq 1$ $\forall s\in [0,t]$,
\begin{equation} \label{eq:Lest}
\int_0^{t}L(s)^{-2}ds\leq \frac{1}{2}\int_1^{t/2-1}s^{2\gamma+2\epsilon-2}ds+4<\frac{1}{2\gamma+2\epsilon-1}t^{2\gamma+2\epsilon-1},
\end{equation}
for $t$ sufficiently large.
Finally, since for $s\in[1,\frac{t}{2}-1]$, $f'(s)=\frac{x}{t}+2\gamma s^{\gamma -1}$  and for $s\in[\frac{t}{2}+1,t-1]$, $f'(s)=\frac{x}{t}-2\gamma (t-s)^{\gamma -1}$, we have 
\begin{align} \label{eq:fest}
\int_0^t f'(s)^2 ds&\leq 4c_1^2+\int_1^{t/2-1}(\tfrac{x}{t}+2\gamma s^{\gamma -1})^2 ds+\int_{1}^{t/2-1}(\tfrac{x}{t}-2\gamma s^{\gamma -1})^2 ds \notag\\
&= 4c_1^2+2\int_1^{t/2-1}(\tfrac{x^2}{t^2}+4\gamma^2 s^{2\gamma -2}) ds \notag\\
&=4c_1^2 +\tfrac{t-4}{t^2}x^2+\tfrac{8\gamma^2}{2\gamma -1}\left(\left(\tfrac{t}{2}-1\right)^{2\gamma-1}-1\right) \notag \\
&<\tfrac{x^2}{t}+\tfrac{8\gamma^2}{2\gamma -1}t^{2\gamma-1}
\end{align}
for $t$ sufficiently large, since $\gamma>1/2$.

Therefore, for any $\delta\in (0,1)$, by Proposition~\ref{prop:roberts} with $p=0$ and $q=\delta/2$, 
\begin{align*}
&\psub{0}{B(s)-f(s)\in (0,L(s))\,\,\forall s\leq t, B(t)-f(t)\in (0,\delta)}\\
&\hspace{0.5cm}
\geq A(Q)\exp\left(-\frac{1}{2}\int_0^t f'(s)^2ds-\frac{\pi^2}{2}\int_0^t \frac{1}{L(s)^{2}}ds \right)
\sin\left(\frac{\pi}{2} \right)
\int_0^{\delta/2} \sin(\pi \nu) d\nu \\
&\hspace{0.5cm}\geq A(Q)\frac{1}{\pi}(1-\cos(\tfrac{1}{2}\pi \delta))\exp\left(-\frac{x^2}{2t}-\frac{4\gamma^2}{2\gamma -1}t^{2\gamma-1}-\frac{\pi^2}{2}\frac{1}{2\gamma+2\epsilon-1}t^{2\gamma+2\epsilon-1} \right)\\
&\hspace{0.5cm}\geq A(Q)\frac{1}{\pi}(1-\cos(\tfrac{1}{2}\pi \delta))\exp\left(-\frac{x^2}{2t}-\frac{\pi^2}{2\gamma+2\epsilon-1}t^{2\gamma+2\epsilon-1} \right),
\end{align*}
where the second inequality holds for $t$ sufficiently large by \eqref{eq:Lest} and \eqref{eq:fest}, and the last inequality is for $t$ sufficiently large.
The result follows by~\eqref{eq:star_heavy}.
\end{proof}
We can now prove Theorem~\ref{thm:heavylower} using Lemma~\ref{lem:abovegamma}.
Assume that $\phi$ satisfies assumption~\eqref{eq:phi_cond}.
Also suppose $u_0\in L^\infty (\R)$, $u_0\geq 0$, $u_0\not\equiv 0$, $\|u_0\|_\infty\leq L$ and $u_0(x)=0$ $\forall x\geq L$, and let $u$ denote the solution of~\eqref{nonlocal_fkpp}.
\begin{prop} \label{prop:heavytaillower}
Suppose that there exist $\alpha\in (0,2)$ and $r_0<\infty$ such that for $r\geq r_0$,
$\int_r^\infty \phi(x)dx \leq r^{-\alpha}.$
Then for $\beta>\frac{2-\alpha}{2+\alpha}$, for $t$ sufficiently large, $u(t,x)\geq m^*(1)$ $\forall x\in [0,\sqrt 2 t -t^\beta]$.
\end{prop}
\begin{proof}
Let $\gamma=\frac{2}{2+\alpha}$; note that $\gamma\in (\frac{1}{2},1)$ and $\alpha \gamma <1$.

By Lemma~\ref{lem:t-1u}, 
for $y\geq 0$ and $t\geq \max(\frac{1}{\sqrt 2 -1}L,1)$, we have
$
u(t,\sqrt 2 t +\frac{1}{2\sqrt 2}\log t +y)\leq L e^{2 L} t^{-1}.
$
Also, by Proposition~\ref{prop:globalbound}, $0\leq u(t,x)\leq M$ $\forall t\geq 0,x\in \R$.

Now take $t$ large.
Suppose that $s\in [\max(\frac{1}{\sqrt 2 -1}L,1,r_0^{1/\gamma}+1),\frac{1}{2}(t-1)]$
and $s$ is sufficiently large that $\frac{1}{2\sqrt 2}\log s \leq (s-1)^\gamma$.
Then if $y\geq \sqrt 2 s+2(s-1)^\gamma\geq \sqrt 2 s +\frac{1}{2\sqrt 2}\log s+(s-1)^\gamma$, we have that
\begin{equation} \label{eq:(1)4}
\phi \ast u(s,y)\leq L e^{2L}s^{-1}\int^{(s-1)^\gamma}_{-\infty} \phi (r)dr+M\int_{(s-1)^\gamma}^\infty \phi (r)dr 
\leq L e^{2L}s^{-1}+M(s-1)^{-\alpha \gamma},
\end{equation}
since $\phi \geq 0$, $\int_{-\infty}^\infty \phi(r)dr=1$ and $(s-1)^\gamma \geq r_0$.
Similarly, for $s\in [\frac{1}{2}(t+3),t-(\log t)^{1/\gamma}]$, then  $(t-s)^\gamma\geq \log t$, and so if $y\geq \sqrt 2 s+2(t-s)^\gamma>\sqrt 2 s+\frac{1}{2\sqrt 2}\log s+(t-s)^\gamma$ then for $t$ sufficiently large,
\begin{equation} \label{eq:(2)4}
\phi \ast u(s,y)
\leq L e^{2L}s^{-1}+M(t-s)^{-\alpha \gamma}.
\end{equation}
Hence if $B(s)\geq \sqrt 2 (t-s)+\min(2s^\gamma, 2(t-1-s)^\gamma)$ $\forall s\in [1,\frac{1}{2}(t-3)]\cup [\frac{1}{2}(t+1),t-2]$,
then by Proposition~\ref{prop:globalbound},
\begin{align} \label{eq:4phiu}
&\int_0^{t-1}\phi \ast u(t-s,B(s))ds \notag\\
&\qquad \leq 2M \left((\log t)^{1/\gamma}+1\right)
+\int_{[(\log t)^{1/\gamma},\frac{1}{2}(t-3)]\cup [\frac{1}{2}(t+1),t-(\log t)^{1/\gamma}]}\phi \ast u(t-s,B(s))ds \notag\\
&\qquad\leq 2M ((\log t)^{1/\gamma}+1)
+\int_{(\log t)^{1/\gamma}}^{t-(\log t)^{1/\gamma}}L e^{2L}s^{-1}ds
+2\int_{(\log t)^{1/\gamma}}^{(t-1)/2}M (s-1)^{-\alpha \gamma}ds \notag\\
&\qquad< 2M ((\log t)^{1/\gamma}+1)
+L e^{2L}\log t
+2M\frac{1}{1-\alpha \gamma}t^{1-\alpha \gamma} \notag\\
&\qquad<\frac{3M}{1-\alpha \gamma}t^{1-\alpha \gamma},
\end{align}
where the second inequality holds for $t$ sufficiently large by~\eqref{eq:(1)4} and~\eqref{eq:(2)4}, the third inequality follows since $\alpha \gamma<1$ and the last line follows
for $t$ sufficiently large.

Take $\epsilon>0$.
As in~\eqref{eq:u>0} in the proof of Theorem~\ref{thm:speed}, note that by the Feynman-Kac formula~\eqref{feynmankac0} and Proposition~\ref{prop:globalbound},
$u(1,\sqrt 2)\geq e^{1-M}\Esub{\sqrt 2}{u_0(B(1))}>0$ since $u_0\geq 0$ and $u_0\not \equiv 0$.
Then by Lemma~\ref{lem:abscty}, for $\delta\in (0,\min(\frac{1}{2}u(1,\sqrt 2),1))$ sufficiently small, we have $u(1,\sqrt 2 +x)\geq \delta$ if $|x|\leq \delta^3$.
By the Feynman-Kac formula~\eqref{feynmankac} and~\eqref{eq:4phiu},
for $t$ sufficiently large,
letting $I=[1,\frac{1}{2}(t-3)]\cup [\frac{1}{2}(t+1),t-2]$,
\begin{align*}
u(t,\sqrt 2 t)
&\geq \delta e^{t-1}e^{-\frac{3M}{1-\alpha \gamma}t^{1-\alpha \gamma}}\\
&\qquad \mathbb P_{\sqrt 2 t}\bigg(|B(t-1)-\sqrt 2|\leq \delta^3,
B(s)\geq \sqrt 2 (t-s)+\min(2s^\gamma, 2(t-1-s)^\gamma)\,\,\forall s\in I\bigg)\\
&= \delta e^{t-1}e^{-\frac{3M}{1-\alpha \gamma}t^{1-\alpha \gamma}}\\
&\qquad\mathbb P_{0}\bigg(|B(t-1)+\sqrt 2 (t-1)|\leq \delta^3,\\
&\qquad\hspace{1cm}B(s)\geq -\sqrt 2(t-1) \tfrac{s}{t-1}+\min(2s^\gamma, 2(t-1-s)^\gamma)\,\,\forall s\in I\bigg)\\
&\geq \delta e^{t-1}e^{-\frac{3M}{1-\alpha \gamma}t^{1-\alpha \gamma}}
c(1-\cos(\tfrac{1}{2}\pi \delta^3))\exp\left(-(t-1)-\frac{\pi^2}{2\gamma+2\epsilon-1}t^{2\gamma+2\epsilon-1} \right)
\end{align*}
for $t$ sufficiently large, where $c=c(\gamma,\epsilon)>0$ by Lemma~\ref{lem:abovegamma}.
Note that since $\gamma=\frac{2}{2+\alpha}$, we have that $1-\alpha \gamma = 2\gamma -1$ and therefore for $t$ sufficiently large,
\begin{align*}
u(t,\sqrt 2 t)
&\geq \exp\left(-t^{2\gamma-1+3\epsilon} \right).
\end{align*}
By Lemma~\ref{lem:growu}, it follows that for $C=C(1/2)$ and $R=R(1/2)$ and for $t$ sufficiently large, there exist $s\in[0,Ct^{2\gamma-1+3\epsilon}]$ and $y\in [-R,R]$ such that
$u(t+s,\sqrt 2 t+y)\geq 1/2$.
Then by Lemma~\ref{lem:travelu}, $\forall s\geq Ct^{2\gamma-1+3\epsilon}+R+t^*(1)$, we have
$u(t+s,\sqrt 2 t)\geq m^*(1)$.

Suppose $t$ is sufficiently large that $2Ct^{2\gamma-1+3\epsilon}>Ct^{2\gamma-1+3\epsilon}+R+t^*(1)$. For $x\in [0,\sqrt 2 t-2\sqrt 2 Ct^{2\gamma-1+3\epsilon}]$,
let $t'=x/\sqrt 2$.
Then $t'<t$ so for $t'$ sufficiently large, for $s\geq Ct^{2\gamma-1+3\epsilon}+R+t^*(1)$ we have
$u(t'+s,\sqrt 2 t')\geq m^*(1)$.
But since $t'\leq t-2Ct^{2\gamma-1+3\epsilon}$ and $x=\sqrt 2 t'$, we have that
$u(t,x)\geq m^*(1)$.

We have now shown that for $t$ sufficiently large, for $x\in [0,\sqrt 2 t-2\sqrt 2 Ct^{2\gamma-1+3\epsilon}]$ sufficiently large, $u(t,x)\geq m^*(1)$.
Finally, as in the proof of Theorem~\ref{thm:speed}, $u(1,0)>0$ and by Lemma~\ref{lem:growu} there exist $s_0\in[0,C\log (1/\min(z_0,u(1,0)))]$ and $y_0\in [-R,R]$ such that $u(1+s_0,y_0)\geq 1/2$.
Then by Lemma~\ref{lem:travelu}, for $t\geq t^*(1)$, for $x\in [y_0-t,y_0+t]$ we have 
$u(1+s_0+t,x)\geq m^*(1)$.
Hence for $t$ sufficiently large,
we have $u(t,x)\geq m^*(1)$ $\forall x\in [0,\frac{1}{2}t]$.
Therefore for $t$ sufficiently large, $u(t,x)\geq m^*(1)$ $\forall x\in [0,\sqrt 2 t-2\sqrt 2 Ct^{2\gamma-1+3\epsilon}]$.
The result follows since $2\gamma-1=\frac{2-\alpha}{2+\alpha}$ and $\epsilon>0$ can be taken arbitrarily small.
\end{proof}
The following result will be used to prove Theorems~\ref{thm:heavyupper} and~\ref{thm:heavyweak}.
\begin{prop} \label{prop:heavytailupper}
Suppose there exist $\alpha \in (0,2)$, $m>0$ and $\gamma'<\gamma:=\frac{2}{2+\alpha}$ such that for $t$ sufficiently large, $u(t,x)\geq m$ $\forall x\in [0,\sqrt 2 t-t^{\gamma'}]$.
Suppose $K<\infty$,
$\beta<\frac{2-\alpha}{2+\alpha}$ and $\epsilon>0$.
For $t$ sufficiently large, 
if $\int_{2t^\gamma}^{2Kt^\gamma} \phi (r)dr\geq 2^{-\alpha} t^{-\alpha \gamma}$
then $u(t,x)<\epsilon$ $\forall x\geq \sqrt 2 t-t^\beta$.
\end{prop}
Before proving this result, we shall prove Theorems~\ref{thm:heavyupper} and~\ref{thm:heavyweak} as corollaries.
We begin by using Proposition~\ref{prop:heavytailupper} to prove Theorem~\ref{thm:heavyweak}.
\begin{cor}
Suppose that there exist $\alpha\in (0,2)$ and $K<\infty$ such that
$
\forall R>0, \exists\, r>R \text{ such that } \int_r^{Kr}\phi (x) dx\geq  r^{-\alpha}.
$
Then for any $\beta<\frac{2-\alpha}{2+\alpha}$ and $\epsilon>0$, for any $T<\infty$ there exist $t\geq T$ and $x\in [0,\sqrt 2 t-t^\beta]$ such that $u(t,x)<\epsilon$.
\end{cor}
\begin{proof}
Take $\beta<\frac{2-\alpha}{2+\alpha}$ and $\epsilon>0$, and suppose, aiming for a contradiction, that for some $T<\infty$, $\forall t\geq T$, $u(t,x)\geq \epsilon$ $\forall x\in [0,\sqrt 2 t-t^\beta]$. Let $\gamma=\frac{2}{2+\alpha}$; there exists $t_0\geq T$ arbitrarily large with $\int_{2t_0^\gamma}^{2Kt_0^\gamma} \phi (r)dr\geq 2^{-\alpha} t_0^{-\alpha \gamma}$. By Proposition~\ref{prop:heavytailupper} with $\gamma'=\beta<\frac{2}{2+\alpha}$ and $m=\epsilon$, it follows that if $t_0$ is sufficiently large then $u(t_0,\sqrt 2t_0-t_0^\beta)<\epsilon$, which is a contradiction. The result follows by increasing $t_0$ if necessary.
\end{proof}
We can also prove Theorem~\ref{thm:heavyupper} as a consequence of Propositions~\ref{prop:heavytaillower} and~\ref{prop:heavytailupper}.
\begin{cor}
Suppose that there exist $\alpha \in (0,2)$, $r_0<\infty$ and $K<\infty$ such that for $r\geq r_0$,
$\int_r^{Kr}\phi (x) dx\geq r^{-\alpha}$
and
$\int_r^{\infty}\phi (x) dx\leq r^{-\alpha/2}.$
Then for any $\beta<\frac{2-\alpha}{2+\alpha}$ and $\epsilon>0$, for $t$ sufficiently large, $u(t,x)<\epsilon$ $\forall x\geq \sqrt 2 t-t^\beta$.
\end{cor}
\begin{proof}
Since $\frac{4-\alpha}{4+\alpha}<\frac{2}{2+\alpha}$, we can take $\gamma'\in(\frac{4-\alpha}{4+\alpha},\frac{2}{2+\alpha})$.
By Proposition~\ref{prop:heavytaillower}, since $\gamma'>\frac{4-\alpha}{4+\alpha}=\frac{2-\frac{1}{2}\alpha}{2+\frac{1}{2}\alpha}$, for $t$ sufficiently large, $u(t,x)\geq m^*(1)$ $\forall x\in [0,\sqrt 2 t -t^{\gamma'}]$.
The result then follows by Proposition~\ref{prop:heavytailupper} with $m=m^*(1)$, since $\gamma'<\gamma:=\frac{2}{2+\alpha}$ and $\int_{2t^\gamma}^{2Kt^\gamma} \phi (r)dr\geq 2^{-\alpha} t^{-\alpha \gamma}$ $\forall t\geq (r_0/2)^{1/\gamma}$.
\end{proof}
It remains to prove Proposition~\ref{prop:heavytailupper}.
\begin{proof}[Proof of Proposition~\ref{prop:heavytailupper}]
Let $\gamma=\frac{2}{2+\alpha}.$
Suppose $t$ is sufficiently large that $t/4>2Kt^\gamma$ and $\forall s \geq t/4$, $u(s,x)\geq m$ $\forall x\in [0,\sqrt 2 s -s^{\gamma'}]$.
If $s\in [t/4,3t/4]$ and $y\in [s,\sqrt 2 s+t^\gamma]$, then 
\begin{align*}
\phi \ast u (s,y)&\geq m \int_{t^\gamma +t^{\gamma'}}^s \phi (r)dr
\geq m \int_{2t^\gamma}^{2Kt^\gamma} \phi (r)dr,
\end{align*}
since $\gamma'<\gamma$ and $s>2Kt^\gamma$.
If $s\in [t/4,3t/4]$ and $y\in [0,s]$ then $\phi \ast u(s,y)\geq m \sigma \eta$, since $\phi\geq \eta$ a.e.~on $(-\sigma,\sigma)$.
It follows that for $t$ sufficiently large, if $B(s)\in [0,\sqrt 2(t-s)+t^\gamma]$ $\forall s\in [t/4,3t/4]$, then
$$
\int_0^t \phi \ast u(t-s,B(s))ds \geq \tfrac{1}{2} t m \min \left(\int_{2t^\gamma}^{2Kt^\gamma} \phi (r)dr,\sigma \eta \right).
$$
Therefore for $t$ sufficiently large, 
if $\int_{2t^\gamma}^{2Kt^\gamma} \phi (r)dr\geq 2^{-\alpha} t^{-\alpha \gamma}$
then
for $x\in \R$,
 by the Feynman-Kac formula~\eqref{feynmankac0} and since $\phi \ast u \geq 0$,
\begin{align} \label{eq:uboundheavy}
u(t,x)&\leq
e^t \Esub{x}{u_0(B(t))\left(e^{-\frac{1}{2}m 2^{-\alpha} t^{1-\alpha \gamma} }+\I{\exists s\in [t/4,3t/4]:B(s)\notin [0,\sqrt 2(t-s)+t^\gamma]}\right)} \notag\\
&\leq Le^{t-\frac{1}{2}m 2^{-\alpha} t^{1-\alpha \gamma}}\psub{x}{B(t)\leq L}\notag \\
&\qquad +Le^t \psub{x}{B(t)\leq L,\exists s\in [t/4,3t/4]:B(s)\notin [0,\sqrt 2(t-s)+t^\gamma]},
\end{align}
since $\|u_0\|_\infty \leq L$ and $u_0(y)=0$ $\forall y\geq L$.
We now want to estimate these two probabilities.

Take $\beta<\frac{2-\alpha}{2+\alpha}$.
Suppose that $x\in [\sqrt 2 t - t^\beta, \sqrt 2 t]$ and let $x_0=\sqrt 2 t-x$.
Then for $y\in [-\sqrt 2 t/8,L]$,
\begin{align*}
&\psub{x}{\exists s\in [t/4,3t/4]:B(s)\notin [0,\sqrt 2(t-s)+t^\gamma]\bigg| B(t)=y}\\
&=\p{\exists s\in [t/4,3t/4]:\xi^t(s)+\tfrac{s}{t}y+\tfrac{t-s}{t}x\notin [0,\sqrt 2(t-s)+t^\gamma]}\\
&=\p{\exists s\in [t/4,3t/4]:\xi^t(s)+\tfrac{s}{t}y-\tfrac{t-s}{t}x_0\notin [-\sqrt 2(t-s),t^\gamma]}\\
&\leq \p{\exists s\in [t/4,3t/4]:|\xi^t(s)|\geq\tfrac{1}{2}t^\gamma}
\end{align*}
for $t$ sufficiently large, since $\beta<1$ and for $s\in [0,t]$, $sy/t \in[-\sqrt 2 t/8,L]$ and $x_0(t-s)/t\in [0, t^\beta]$ so $\tfrac{s}{t}y-\tfrac{t-s}{t}x_0\in [-\sqrt 2 t/8-t^\beta, L]$.
By Brownian scaling, and then since $(\xi^1(s),0\leq s \leq 1)\eqdist (B(s)-sB(1),0\leq s\leq 1)$, it follows that
\begin{align*}
&\psub{x}{\exists s\in [t/4,3t/4]:B(s)\notin [0,\sqrt 2(t-s)+t^\gamma]\bigg| B(t)=y}\\
&\leq \p{\exists s\in [0,1]:|\xi^1(s)|\geq\tfrac{1}{2}t^{\gamma-\frac{1}{2}}}\\
&= \psub{0}{\exists s\in [0,1]:|B(s)-sB(1)|\geq\tfrac{1}{2}t^{\gamma-\frac{1}{2}}}\\
&\leq \psub{0}{\sup_{s\in [0,1]}|B(s)|\geq\tfrac{1}{4}t^{\gamma-\frac{1}{2}}}.
\end{align*}
By the reflection principle, it follows that
\begin{align*}
&\psub{x}{\exists s\in [t/4,3t/4]:B(s)\notin [0,\sqrt 2(t-s)+t^\gamma]\bigg| B(t)=y}\\
&\leq 4\psub{0}{B(1)\geq\tfrac{1}{4}t^{\gamma-\frac{1}{2}}}\\
&\leq 4\exp(-t^{2\gamma-1}/32)
\end{align*}
by \eqref{eq:gaussiantail}.
Therefore,
\begin{align} \label{eq:(*)4}
&\psub{x}{B(t)\leq L,\exists s\in [t/4,3t/4]:B(s)\notin [0,\sqrt 2(t-s)+t^\gamma]} \notag\\
&\qquad\leq 4e^{-t^{2\gamma-1}/32}\psub{x}{B(t)\in [-\sqrt 2 t/8, L]}
+\psub{x}{B(t)\leq -\sqrt 2 t/8}.
\end{align}
Now by \eqref{eq:gaussiantail}, since $x\geq \sqrt 2 t-t^\beta$,
\begin{align*}
\psub{x}{B(t)\leq L}&\leq 
\exp(-(\sqrt 2 t-t^\beta -L)^2/2t)
\leq \exp(-t+\sqrt 2 t^\beta +\sqrt 2 L).
\end{align*}
Also by \eqref{eq:gaussiantail},
$$
\psub{x}{B(t)\leq -\sqrt 2 t/8}\leq \exp\left(-\tfrac{1}{2t}\left(\tfrac{9\sqrt 2}{8}t - t^\beta\right)^2\right)
\leq e^{-5t/4}
$$
for $t$ sufficiently large.
Therefore, substituting into \eqref{eq:(*)4} and then~\eqref{eq:uboundheavy}, for $t$ sufficiently large, for $x\in [\sqrt 2 t - t^\beta, \sqrt 2 t]$, if $\int_{2t^\gamma}^{2Kt^\gamma} \phi (r)dr\geq 2^{-\alpha} t^{-\alpha \gamma}$ then 
\begin{align*}
u(t,x)
&\leq Le^{t-\frac{1}{2}m 2^{-\alpha} t^{1-\alpha \gamma}}e^{-t+\sqrt 2 t^\beta +\sqrt 2 L}  +Le^t ( 4e^{-t^{2\gamma-1}/32}e^{-t+\sqrt 2 t^\beta +\sqrt 2 L}+e^{-5t/4})\\
&\leq Le^{\sqrt 2 L}\left(e^{\sqrt 2 t^\beta-\frac{1}{2}m 2^{-\alpha} t^{1-\alpha \gamma}} +4e^{\sqrt 2 t^\beta-t^{2\gamma-1}/32}\right)+Le^{-t/4}.
\end{align*}
Since $\beta<\frac{2-\alpha}{2+\alpha}=1-\alpha \gamma=2\gamma-1$, for $\epsilon>0$ fixed, if $t$ is sufficiently large and $\int_{2t^\gamma}^{2Kt^\gamma} \phi (r)dr\geq 2^{-\alpha} t^{-\alpha \gamma}$ then $u(t,x)<\epsilon$ $\forall x\in [\sqrt 2 t - t^\beta, \sqrt 2 t]$.
By \eqref{uabovesqrt2} in the proof of Theorem~\ref{thm:speed}, for $t$ sufficiently large, $u(t,x)<\epsilon$ $\forall x\geq \sqrt 2 t$. The result follows.

\end{proof}

\small 


\end{document}